\documentclass[letterpaper]{amsart}                   

\usepackage{amsmath,amsfonts,amssymb,amscd,amsthm,bbm}

\usepackage{ytableau,tikz}

\newtheorem{thm}{Theorem}[section]
\newtheorem{prop}[thm]{Proposition}
\newtheorem{cor}[thm]{Corollary}

\theoremstyle{definition}
\newtheorem{rem}[thm]{Remark}

\newcommand{\be}{\begin{equation}}  \newcommand{\ee}{\end{equation}}
\newcommand{\bea}{\begin{eqnarray*}}  \newcommand{\eea}{\end{eqnarray*}}

\begin{document}

\newcommand{\C}{{\mathbb C}}
\newcommand\R{{\mathbb R}}
\newcommand\Z{{\mathbb Z}}
\newcommand\N{{\mathbb N}}

\newcommand\F{{\mathcal F}}
\renewcommand\H{{\mathcal H}}
\newcommand\I{{\mathcal I}}
\newcommand\J{{\mathcal J}}

\renewcommand\l{\lambda}
\newcommand\G{\Gamma}

\newcommand\E{{\mathbb E}}
\renewcommand\P{{\mathbb P}}

\newcommand\rpp{{\rm RPP}}

\title{Discrete Whittaker processes}
\author{Neil O'Connell}

\thanks{Research supported by the European Research Council (Grant number 669306)}
\address{\parbox{\linewidth}{School of Mathematics and Statistics\\ University College Dublin\\ Dublin 4, Ireland}}

\begin{abstract}
We consider a Markov chain on non-negative integer arrays of a given shape (and satisfying certain constraints)
which is closely related to fundamental $SL(r+1,\R)$ Whittaker functions and the Toda lattice.  
In the index zero case the arrays are reverse plane partitions. We show that this Markov chain has 
non-trivial Markovian projections and a unique entrance law starting from the array with all entries 
equal to~$+\infty$.  We also discuss connections with imaginary exponential functionals of Brownian motion,
a semi-discrete polymer model with purely imaginary disorder,
interacting corner growth processes and discrete $\delta$-Bose gas,
extensions to other root systems, and hitting probabilities for some low rank examples.
\end{abstract}

\maketitle

\thispagestyle{empty}

\section{Introduction}

\subsection{Fundamental Whittaker functions}

Fundamental $SL(r+1,\R)$ Whittaker functions are series solutions to the eigenvalue equation 
\be\label{eeH} \H^r\phi=-\l^2\phi,\ee where
\be\label{HT}
\H^r = -\frac12 \sum_{i=1}^{r+1} \frac{\partial^2}{\partial x_i^2} + \sum_{i=1}^{r}e^{x_i-x_{i+1}}
\ee
is the Hamiltonian of the $(r+1)$-particle quantum Toda chain.  They were first introduced by
Hashizume~\cite{h} in a more general context.

Let $\alpha_1,\alpha_2,\ldots$ be a fixed sequence of complex numbers (independent of $r$)
and denote $\alpha_{ij}=\alpha_i+\alpha_{i+1}+\cdots+\alpha_j$ for $i\le j$. 
Let $r\ge 1$ and define $\nu\in\C^{r+1}$ by $\sum_{i=1}^{r+1}\nu_i=0$ 
and $\alpha_i=\nu_i-\nu_{i+1}$ for $i=1,\ldots,r$.
Then
\begin{equation} \label{phir}
\phi_r(x)=\sum_{n\in\Z_+^r} a_{r}(n) \prod_{i=1}^r e^{n_i(x_i-x_{i+1})+\nu_i x_i}
\end{equation}
satisfies \eqref{eeH} with $\l^2=\sum_{i=1}^{r+1}\nu_i^2/2$
provided the coefficients $a_{r}(n)$ satisfy
\be\label{rec} \left[\sum_{i=1}^{r} n_i^2 - \sum_{i=1}^{r-1}n_i n_{i+1} +\sum_{i=1}^{r} \alpha_i n_i \right] 
a_{r}(n) =  \sum_{i=1}^{r} a_{r}(n-e_i),\ee
where $e_1,\ldots,e_r$ denote the standard basis vectors in $\Z^r$ and
with the convention \hbox{$a_{r}(n) =0$} for $n\notin \Z_+^{r}$.
Ishii and Stade~\cite{is} obtained the following recursive formula for the coefficients~$a_{r}(n)$.  
For $n\in \Z_+^{r}$ and $k\in \Z_+^{r-1}$, define
\be\label{L}
q_{r}(n,k)=\prod_{i=1}^{r} \frac1{(n_i-k_i)! (n_i-k_{i-1}+\alpha_{ir})!} ,
\ee
with the convention $k_0=k_{r}=0$.
Define 
$a_{r}(n)$, $r\ge 1$, $n\in \Z_+^{r}$, recursively by
\be\label{def-a1} a_{1}(n) =\frac1{n! (n+\alpha_1)!}\ee
and, for $r\ge 2$,
\be\label{def-ar}
a_{r}(n) =\sum_{k} q_{r}(n,k) a_{r-1}(k)  .
\ee
where the sum is over $0\le k_i\le n_i$, $i=1,\ldots,r-1$.
Then, for each $r\ge 1$, $a_{r}(n)$ satisfies the difference equation \eqref{rec} with 
\be a_{r}(0) =\prod_{1\le i\le j\le r}\frac{1}{\alpha_{ij}!}.\ee

For example, 
$$\phi_1(x_1,x_2)=I_{\alpha_1}\left(2e^{(x_1-x_2)/2}\right),$$
where $I_a(z)$ is the modified Bessel function of the first kind
$$I_a(z)=\sum_{n=0}^\infty \frac1{n! (n+a)!} \left(\frac{z}{2}\right)^{2n+a}.$$
For $r=2$, the coefficients are given by
\be\label{bf}
a_{2}(n,m)=\frac{(n+m+a+b)!}{n!(n+a)!(n+a+b)! m!(m+b)! (m+a+b)!},
\ee
where $a=\alpha_1$ and $b=\alpha_2$.  This formula is due to Bump~\cite{bump}.

We remark that, in the case $\alpha\equiv 0$, the recursive formula \eqref{def-ar} 
agrees with a special case (for complete flag manifolds) of a formula given in Batyrev, 
Ciocan-Fontanine, Kim and van Straten~\cite[Theorem 5.1.6]{bcks} for the coefficients of hypergeometric 
series of partial flag manifolds (see also Remark~\ref{bat} below).

The formula \eqref{def-ar} is quite similar to Givental's integral formula~\cite{giv,jk,gklo} for 
another family of eigenfunctions of the quantum Toda lattice known as class one Whittaker functions. 
It is this similarity which motivated the present work.  
Gerasimov, Kharchev,  Lebedev and Oblezin~\cite{gklo}
showed that Givental's formula may be understood in terms of some intertwining relations
between the Hamiltonians of quantum Toda chains with different numbers of particles.  
These intertwining relations were extended and given a probabilistic interpretation,
in terms of Brownian motion, in~\cite{noc12}.  In this paper we will show that Ishii and Stade's formula \eqref{def-ar}
admits a similar development, in a discrete setting.

\subsection{Reverse plane partitions}

When $\alpha\equiv 0$, the natural setting for this is in the context of reverse plane partitions.
In order to describe this, we briefly recall some combinatorial definitions.
For more background, see for example~\cite{st}.  An {\em integer partition} $\l$ is a 
sequence of nonnegative integers $\l_1\ge\l_2\ge\cdots$ with $\sum_i \l_i<\infty$.
An integer partition $\l$ may be represented by its {\em Young diagram}, which is a left-justified
array of boxes, with $\l_1$ boxes in the first row, $\l_2$ in the second, and so on.
For example, the integer partition $\l=(4,3,1)$ has Young diagram

$$
\l=(4,3,1)=\ydiagram{4,3,1}
$$

\medskip\noindent
We may also identify $\l$ with the set $\{(i,j)\in\N^2:\ 1\le j\le \l_i\}$, so that $(i,j)\in\l$
refers to the box in the $i^{th}$ row (from top) and $j^{th}$ column (from left) of $\l$.
The non-zero $\l_i$ are referred to as the {\em parts} of $\l$, and the partition with no parts
is called the empty partition and denoted $\l=\emptyset$.
The partitions $\delta_{r+1}=(r,r-1,\ldots,1)$, where $r\ge 1$, are called {\em staircase shapes}.

For two partitions $\mu$ and $\l$, we write $\mu\subset\l$ if the diagram of $\mu$
is contained in that of $\l$.  If $\mu\subset\l$, then the difference 
$$\l/\mu=\{(i,j)\in\l:\ (i,j)\notin \mu\}$$ 
is called a {\em skew partition/diagram}.  
For example, 

$$(4,3,1)/(2,1) = \ydiagram{2+2,1+2,1}$$

\medskip\noindent
Note that if $\mu=\emptyset$, then $\l/\mu=\l$.

For $\l$ a partition, we will denote by $\l^\circ\subset\l$ the 
set of $(i,j)\in\l$, such that $(i+1,j)\in\l$ and $(i,j+1)\in\l$.
For example, if $\l=(4,3,1)$ then $\l^\circ=(3,1)$ is the shaded region
$$\begin{ytableau}
*(gray)&*(gray)&*(gray)&\\
*(gray)&&\\
\\
\end{ytableau}$$

\medskip\noindent
We remark that the notation $\l^\circ$ is not standard, but convenient in our context.

Given an integer partition $\l$, a {\em reverse plane partition $\pi$ with shape $\l$} is a filling of $\l$ 
with non-negative integers $(\pi_{ij},\ (i,j)\in\l)$ which is weakly increasing across rows and down columns. 
For example,

\ytableausetup{boxsize=1.8em}
$$\pi = \begin{ytableau}
\pi_{11} & \pi_{12}&\pi_{13}\\
\pi_{21}&\pi_{22}\\
\pi_{31}\\
\end{ytableau}
=\begin{ytableau}
0&1&3\\
1&1\\
2\\
\end{ytableau}$$
\ytableausetup{boxsize=normal}

\medskip\noindent
is a reverse plane partition with staircase shape $\delta_4=(3,2,1)$.
More generally, a {\em reverse plane partition $\pi$ with shape $\l/\mu$} is a filling of $\l/\mu$ 
with non-negative integers $(\pi_{ij},\ (i,j)\in\l/\mu)$ which is weakly increasing across rows and down columns. 
If $\pi$ is a reverse plane partition with shape $\l$ and $\mu\subset\l$, we denote by 
$$\pi |_{\l/\mu} = (\pi_{ij},\ (i,j)\in\l/\mu)$$ the restriction of $\pi$ to $\l/\mu$,
which is itself a reverse plane partition of shape $\l/\mu$.

\subsection{Discrete Whittaker processes}

Let $\rpp(\l)$ denote the set of reverse plane partitions of shape $\l$.  
Fix $\l$, and consider the (continuous time) Markov chain on $\rpp(\l)$,
defined as follows: for each $(i,j)\in\l$, subtract one from $\pi_{ij}$ at rate 
$$b_{ij}(\pi)=(\pi_{ij}-\pi_{i,j-1})(\pi_{ij}-\pi_{i-1,j}),$$
with the convention $\pi_{i,0}=\pi_{0,j}=0$.  
The infinitesimal generator of this Markov chain is given by the difference operator
$$G^\l = \sum_{(i,j)\in\l} b_{ij}(\pi) D_{\pi_{ij}},$$
where $D_n$ denotes the backward difference operator $D_nf(n)=f(n-1)-f(n)$.

For example, if $\l=(2,1)$ and the current state is 
the reverse plane partition $\pi_{11}=2$, $\pi_{12}=4$, $\pi_{21}=3$ 
then the chain jumps to one of the three adjacent
reverse plane partitions, according to the rates indicated:

$$\begin{ytableau} 1&4\\3\\ \end{ytableau} \quad \stackrel{4}{\longleftarrow} \quad
\begin{ytableau} 2&4\\3\\ \end{ytableau} \quad \stackrel{8}{\longrightarrow} \quad
\begin{ytableau} 2&3\\3\\ \end{ytableau}$$
$$\Big\downarrow \scriptstyle{3}$$
$$\begin{ytableau} 2&4\\2\\ \end{ytableau}$$

\medskip
Note that if $\pi$ is a Markov
chain on $\rpp(\l)$ with generator $G^\l$ and $\mu\subset\l$, then the restriction
of $\pi$ to $\mu$ is a Markov chain on $\rpp(\mu)$ with generator $G^\mu$.  
In particular, the first row of $\pi$ is a Markov chain in its own right,
and it is natural to think of it as an interacting particle system on the non-negative integers: 
the values $n_j:=\pi_{1j},\ j=1,\ldots,\l_1$ are the positions of $\l_1$ particles;
the left-most particle at position $n_1$ jumps to the left at rate $n_1^2$, while for each $j>1$, 
the particle at position $n_j$ jumps to the left at rate $n_j(n_j-n_{j-1})$.

We may also consider restrictions to certain skew diagrams $\l/\mu$.
For this we require that $\mu\subset\l^\circ$.
Remarkably, if the initial law on $\rpp(\l)$ is chosen correctly,
then the restriction $\pi |_{\l/\mu}$ will evolve as a Markov chain in its own right.

The simplest non-trivial example is related to Vandermonde's identity
$${n+m\choose n}=\sum_k {n\choose k} {m\choose k}.$$ 
Let $\l=(2,1)$, $\mu=(1)$, and write 
$\pi_{11}=k$, $\pi_{12}=n$, $\pi_{21}=m$.
In this notation,
$$G^{(2,1)}=k^2 D_k + n(n-k) D_n + m(m-k) D_m.$$
Suppose that, at time zero, $\pi_{12}=n$, $\pi_{21}=m$ and
$\pi_{11}$ is chosen at random according to the probability distribution
$$
p_{n,m}(k)={n+m\choose n}^{-1} {n\choose k} {m\choose k},\qquad 0\le k\le n\wedge m.
$$
Then, if $\pi$ evolves according to $G^{(2,1)}$, the restriction $\pi |_{\l/\mu}=(\pi_{12}, \pi_{21})$ 
is also a Markov chain, in its own filtration, with generator
$$L = \frac{n^3}{n+m} D_{n} + \frac{m^3}{n+m}D_{m}.$$

More generally, if $\l$ is the staircase shape $\delta_{r+1}=(r,r-1,\ldots,1)$ and $\mu=\delta_r$,
then the restriction $\pi |_{\l/\mu}$ represents the `boundary values' $n_i=\pi_{i,r-i+1}$, $i=1,\ldots,r$.
In this setting, we will prove the following.
\begin{thm}
Suppose that, initially, the conditional law of $\pi |_{\delta_r}$, given these boundary values,
is proportional to
\be\label{meas}
W_r(\pi)=\prod_{(i,j)\in\delta_r} {\pi_{i,j+1} \choose \pi_{ij}} {\pi_{i+1,j} \choose \pi_{ij}}.
\ee
Then, if $\pi$ evolves according to $G^{\delta_{r+1}}$, the boundary values
$(\pi_{1,r},\ldots,\pi_{r,1})$ will evolve as a Markov chain on $\Z_+^{r}$ with generator 
\be\label{Lr0}
L^{r}= \sum_{i=1}^r \frac{a_r(n-e_i)}{a_r(n)} D_{n_i},
\ee
where $a_r(n)$ are the series coefficients defined by \eqref{def-ar} with $\alpha\equiv 0$.
\end{thm}
A more precise statement (for more general $\alpha$) is given in Theorem~\ref{mf-nu} below.

We will also show that the Markov chain on $\rpp(\l)$ with generator $G^\l$ has a unique
entrance law starting from $\pi_{ij}=+\infty$ for all $(i,j)\in\l$.  An important ingredient for proving 
this is a law of large numbers (via a large deviation principle) for the distribution \eqref{meas}, 
as the boundary values go to infinity. 

\subsection{Dualities and imaginary exponential functionals of Brownian motion}

When $\alpha=0$, the difference equation \eqref{rec} may be written as $h^{r} a_{r}=0$, where
\be\label{hr0}
h^{r} = \sum_{i=1}^r L_{n_i} - \sum_{i=1}^{r} n_i^2 + \sum_{i=1}^{r-1}n_i n_{i+1} .
\ee
The generator $L^r$ defined by \eqref{Lr0} is the corresponding Doob transform 
$$L^r=a_r(n)^{-1}\circ h^r \circ a_r(n).$$
The difference operator $h^r$ is related to the quantum Toda Hamiltonian $\H^r$, defined by \eqref{HT},
via the duality relation 
$$h_n^r f(n,x) = \H^r_x f(n,x),\qquad f(n,x)=\prod_{i=1}^r e^{n_i(x_{i+1}-x_i)}.$$
Equivalently, $L^r_n h(n,x) = \H^r_x h(n,x)$, where $h(n,x)=a_r(n)^{-1} f(n,x)$.  
Changing variables to $x_k=i\xi_k$, this becomes a {\em Markov} duality relation, namely
$$L^r_n H(n,\xi) = \I^r_\xi H(n,\xi)$$
where $H(n,\xi)=h(n,i\xi)$ and
$$\I^r=\frac12\sum_{j=1}^{r+1} \frac{\partial^2}{\partial \xi_j^2}+\sum_{k=1}^r e^{i(\xi_k-\xi_{k+1})}.$$
This Markov duality provides a relation between the transition probabilities of the Markov
chain with generator $L^r$ and certain imaginary exponential functionals of Brownian motion,
as follows.
Denote by $p_t(n,m)$ the transition kernel of the Markov chain on $\Z_+^r$ with generator $L^r$,
$\P_n$ the law of this process started from $n\in\Z_+^r$,
and by $T_0$ the first hitting time of $0\equiv (0,0,\ldots,0)$.
Let $B=(B_1,\ldots,B_{r+1})$ be a standard Brownian motion in $\R^{r+1}$,
started at the origin, and set
$$Y_k(t)=e^{i(B_k(t)-B_{k+1}(t))},\qquad Z_k(t)=\int_0^t Y_k(s) ds,\qquad k=1,\ldots, r.$$
For $0\ne y,z\in\C^r$ and $n\in\Z_+^r$, write
$$y^{-n}=\prod_{k=1}^r y_k^{-n_k},\quad z^n=\prod_{k=1}^r z_k^{n_k},\quad n!=\prod_{k=1}^r n_i! .$$
Then (see Proposition~\ref{ief1} below) for $n,m\in\Z_+^r$ with $n\ge m$,
$$p_t(n,m)=\frac{a_r(m)}{a_r(n)}  \frac1{(n-m)!} \E  \left[Y(t)^{-n} Z(t)^{n-m}\right].$$
In particular, since $Y(t)^{-1}Z(t)$ has the same law as $Z(t)$, this implies
$$\P_n(T_0\le t) \equiv p_t(n,0) =  \frac{\E Z(t)^{n}}{a_r(n) n!}.$$
Letting $t\to\infty$, this also yields a probabilistic interpretation of the (index zero) fundamental 
$SL(r+1,\R)$ Whittaker function.  
When $\alpha\equiv 0$, the fundamental Whittaker function
$\phi_r(x)$ defined by \eqref{phir} may be written as $\phi_r(x)=\Phi_r(y)$ where
$y_i=e^{x_i-x_{i+1}}$, $i=1,\ldots,r$, and
$$\Phi_r(y) = \sum_{n\in\Z_+^r} a_r(n) y^n.$$
This series is well defined for any $y\in\C^r$, and satisfies
(see Corollary~\ref{ief2} below):
$$\Phi_r(y) = \lim_{t\to\infty} \E \exp\left(\sum_{k=1}^r y_k Z_k(t)\right) .$$
This complements a similar representation for class one Whittaker functions in 
terms of exponential functionals of Brownian motion given in \cite[Proposition 5.1]{boc}.

On the other hand, recall that the first row $n_i=\pi_{1,r-i+1}$, $i=1,\ldots,r$
of the Markov chain with generator $G^{\delta_{r+1}}$, evolves as a Markov
chain on $$E_r=\{n\in\Z_+^r:\ n_1\ge\cdots\ge n_r\ge 0\}$$ with generator
$$M^r= \sum_{i=1}^{r-1} n_i (n_i-n_{i+1}) D_{n_i} +n_r^2 D_{n_r}.$$
Denote the transition kernel of this Markov chain by $q_t(n,m)$.
In this case we have the Markov duality relation
$$M_n^r K(n,\xi) = \J_\xi^r K(n,\xi),$$ where, writing $\xi=a+ib$ with $a,b\in\R^{r+1}$,
$$\J^r = \frac12\sum_{j=1}^{r+1} \frac{\partial^2}{\partial a_j^2}
+\sum_{k=1}^r \Re\left[ e^{i(\xi_k-\xi_{k+1})} \right] \frac{\partial}{\partial a_{k+1}} 
+\sum_{k=1}^r \Im\left[ e^{i(\xi_k-\xi_{k+1})} \right] \frac{\partial}{\partial b_{k+1}} $$
and
$$K(n,\xi)= \prod_{k=1}^r n_k! e^{i n_k (\xi_{k+1}-\xi_k)} .$$
It follows (see Proposition~\ref{ief3} below) that if we define
$$U^r(t) = \int_{0<s_1<s_2<\cdots<s_r<t} Y_1(s_1)\ldots Y_r(s_r) ds_1\ldots ds_r,$$
then
$$q_t(p^r,0) = \frac1{p!^r} \E U^r(t)^p,$$
where $p\in\Z_+$ and $p^r=(p,p,\ldots,p)\in E_r$.
In particular, we deduce the following (see Corollary~\ref{ief4} below):
if $(N_1(t),\ldots,N_r(t)),\ t>0$ is a Markov chain with generator $L^r$ started from $+\infty$
and $\zeta=\inf\{ t>0:\ N_1(t)=0\}$, then
$$\P(\zeta\le t) = \lim_{p\to\infty} \frac1{p!^r} \E U^r(t)^p.$$

\begin{rem}
In the paper~\cite{noc12}, a continuous time version of geometric RSK correspondence
was considered, with Brownian motion as input, and related to the quantum Toda lattice.
In that setting, the class one Whittaker functions play a central role and an important application 
is to the semi-discrete polymer in a Brownian environment introduced in~\cite{oy}.
The random variable $U^r(t)$ is essentially the partition function for the
corresponding random polymer model with purely imaginary disorder, and arises similarly
in the context of the continuous time geometric RSK correspondence with imaginary
Brownian motion as input.  It would be interesting to
understand if there are further parallels with the results of~\cite{noc12} and subsequent
related works on variations of the geometric RSK correspondence with random input
and Whittaker processes, see for example~\cite{bc,ch1,ch2,cosz,noc13,noc14,osz}.
\end{rem}

\subsection{Interacting corner growth processes and discrete $\delta$-Bose gas}

The Markov chain on $\rpp(\l)$ with generator $G^\l$ has an alternative description as a 
system of interacting corner growth processes, and the evolution of the first row is closely
related to a discrete repulsive $\delta$-Bose gas.  For this description it makes sense to 
let $\l=\N^2$, in a sense which will soon be explained.  Let us adopt 
the following convention: if $\mu$ is a partition with $l$ parts, then $\mu_0=+\infty$ 
and $\mu_{i}=0$ for $i>l$.

The corner growth process was first introduced and studied by Johansson~\cite{j}.
It is closely related to the totally asymmetric exclusion process on the integers.
It may be described as a continuous time Markov chain on the set of integer partitions, 
which evolves as follows: if the current state is $\mu$ then for all $i$ such that $\mu_i<\mu_{i-1}$,
add 1 to $\mu_i$ at rate 1.  In other words, everywhere a box can be added to $\mu$ so that the 
resulting diagram is a partition, it is added at rate 1.  For example, if the current state is $\mu=(2,2,1)$,
then any one of the shaded boxes shown below may be added to $\mu$, and each does so with rate 1:

\bigskip

\begin{center}
\begin{ytableau}
\ & &*(gray)\\
& \\
&*(gray)\\
*(gray)\\
\end{ytableau}
\hspace{1.5cm}
\begin{tikzpicture}[scale=0.54,baseline={([yshift=5.4ex]current bounding box)}]
\draw (0,0) -- (0,2) -- (1,2) -- (1,3) -- (2,3) -- (2,5) -- (4,5);
\draw [dotted] (0,1) -- (1,1) -- (1,2);
\draw [dotted] (2,4) -- (3,4) -- (3,5);
\draw [dotted] (1,2) -- (2,2) -- (2,3);

\draw [-latex] (.2,1.8) -- (.8,1.2);
\draw (.62,1.72) node {\tiny 1};

\draw [-latex] (1.2,2.8) -- (1.8,2.2);
\draw (1.62,2.72) node {\tiny 1};

\draw [-latex] (2.2,4.8) -- (2.8,4.2);
\draw (2.62,4.72) node {\tiny 1};
\end{tikzpicture}
\end{center}

\bigskip\noindent The representation on the right will be helpful for later reference.

A reverse plane partition $\pi\in\rpp(\l)$ may be identified with a nested sequence of partitions 
$\mu^0\subset\mu^1\subset\cdots$, where 
$$\mu^k = \{(i,j)\in\l:\ \pi_{ij}\le k\}.$$
Let us define $\rpp(\N^2)$ to be the set of integer arrays $(\pi_{ij},\ (i,j)\in\N^2)$
which satisfy $\pi_{ij}\ge \pi_{i,j-1} \vee \pi_{i-1,j}$ for $(i,j)\in\N^2$,
with the convention $\pi_{ij}=0$ for $(i,j)\notin\N^2$, and for which
$\mu^k =  \{(i,j)\in\l:\ \pi_{ij}\le k\}$, $0\le k< \sup_{i,j}\pi_{ij}$, defines a 
(possibly infinite) nested sequence of integer partitions.
As in the finite case, we will identify an element $\pi\in\rpp(\N^2)$ with the
corresponding nested sequence of partitions $\mu^0\subset\mu^1\subset\cdots$.
Let us also adopt the following conventions: 
if $\mu$ is a partition with $l$ parts, then $\mu_0=+\infty$ and $\mu_{i}=0$ for $i>l$ (as above);
if $\pi\in\rpp(\N^2)$ and $N=\sup_{i,j}\pi_{ij}<\infty$, so that
the associated nested sequence $\mu^0\subset\mu^1\subset\cdots\subset\mu^{N-1}$
is finite, then $\mu^N_i=+\infty$ for all $i\ge 0$.

Recall that if $\pi$ evolves as a Markov chain with generator $G^\l$ and $\mu\subset\l$,
then the restriction of $\pi$ to $\mu$ evolves as a Markov chain with generator $G^\mu$.
Given this property, by Kolmogorov consistency,
we may define a Markov chain on $\rpp(\N^2)$ such that its restriction 
to any given shape $\l$ is a Markov chain with generator~$G^\l$.  The evolution of
this process, in terms of the $\mu^k$, is as follows.  
If the current state of the Markov chain is $\pi = (\mu^0,\mu^1,\ldots)$ then,
for each $i$ and $k$ such that $\mu^k_i<\mu^k_{i-1}\wedge\mu^{k+1}_i$,
add 1 to $\mu^k_i$ at rate given by the product 
$$\# \{0\le j\le k:\ \mu^j_i=\mu^k_i\} \times \# \{ 0\le j\le k:\ \mu^j_{i-1} > \mu^k_i\}.$$
For example, if $N=3$ and the current state is 
$$\mu^0=(1,1) \subset \mu^1=(4,3) \subset \mu^2=(4,3,3,3),$$ 
then the possible transitions are as shown below with rates indicated:
\medskip
\begin{center}
\begin{tikzpicture}[scale=0.54]
\draw (0,0) -- (0,5) -- (1,5) -- (1,7) -- (7,7);
\draw [dotted] (1,6) -- (2,6) -- (2,7);
\draw [shift={(.1,-.1)}] (0,0) -- (0,5) -- (3,5) -- (3,6) -- (4,6) -- (4,7) -- (7,7);
\draw [shift={(.1,-.1)}, dotted] (0,4) -- (1,4) -- (1,5);
\draw [shift={(.2,-.2)}] (0,0) -- (0,3) -- (3,3) -- (3,6) -- (4,6) -- (4,7) -- (7,7);
\draw [shift={(.2,-.2)}, dotted] (0,2) -- (1,2) -- (1,3);
\draw [shift={(.2,-.2)}, dotted] (3,5) -- (4,5) -- (4,6);
\draw [shift={(.2,-.2)}, dotted] (4,6) -- (5,6) -- (5,7);

\draw [-latex] (1.2,6.8) -- (1.8,6.2);
\draw (1.62,6.72) node {\tiny 1};

\draw [shift={(-.9,-2.1)}] [-latex] (1.2,6.8) -- (1.8,6.2);
\draw [shift={(-.9,-2.1)}] (1.62,6.72) node {\tiny 4};

\draw [shift={(3.2,-.2)}] [-latex] (1.2,6.8) -- (1.8,6.2);
\draw [shift={(3.2,-.2)}] (1.62,6.72) node {\tiny 6};

\draw [shift={(2.2,-1.2)}] [-latex] (1.2,6.8) -- (1.8,6.2);
\draw [shift={(2.2,-1.2)}] (1.62,6.72) node {\tiny 4};

\draw [shift={(-.8,-4.2)}] [-latex] (1.2,6.8) -- (1.8,6.2);
\draw [shift={(-.8,-4.2)}] (1.62,6.72) node {\tiny 3};

\end{tikzpicture}
\end{center}

\noindent
In this graphical illustration, the rates shown are given by the product of the number
of lines immediately above and the number of lines immediately to the left.

Note that the value of $N$ is preserved under these dynamics and,
if $N=1$, then the evolution of $\mu^0$ is precisely that of the corner growth process.

If $N<\infty$, the evolution of the first row is related to a discrete $N$-particle replusive
$\delta$-Bose gas on $\Z_+$, with Hamiltonian 
$$H = \sum_{i=1}^N B_{x_i} - \sum_{1\le i<j\le N} \mathbbm{1}_{x_i=x_j}.$$
In this expression, $B_a$ denotes the forward difference operator $$B_a f(a)=f(a+1)-f(a),$$
$\mathbbm{1}_{a=b}$ is the indicator function which takes the value 1 if $a=b$
and 0 otherwise, and $x=(x_1,\ldots,x_N)$ are the locations of $N$ indistinguishable 
particles on $\Z_+$.
The Hamiltonian $H$ acts on functions which are symmetric in the
variables $x_1,\ldots,x_N$.  It has a positive eigenfunction given by
$$\varphi(x) = \prod_{i=1}^N i^{x_{(i)}},$$
where $x_{(1)}\le x_{(2)}\le \cdots \le x_{(N)}$ denotes the weakly increasing rearrangement 
of $x_1,\ldots,x_N$.  The corresponding eigenvalue is $N(N-1)/2$, and the
corresponding Doob transform is given by
$$R=\varphi(x)^{-1}\circ (H-N(N-1)/2) \circ \varphi(x) = \sum_{i=1}^N r_i(x) B_{x_i} $$
where 
$$r_i(x)=\#\{1\le k\le N:\ x_k\le x_{(i)} \}$$
is the `rank' of the particle(s) at location $x_{(i)}$.
This is the generator of $N$ Poisson counting processes on $\Z_+$ with rank-dependent rates.
Similar processes, for example one dimensional Brownian motions with rank dependent drifts, 
are well studied in the literature, see for example~\cite{ld,pp}.
Returning to the above Markov chain on $\rpp(\N^2)$, observe that the evolution
of the first row, that is, the set of locations $\mu^0_1\le \mu^1_1\le \cdots\le \mu^{N-1}_1$,
agrees precisely with the process with generator $R$.
We remark that, given this connection, the connection with the semi-discrete
polymer with purely imaginary disorder described earlier is not too surprising.
See for example \cite[Section 6]{bc} for a detailed discussion on the connection
between the semi-discrete polymer and discrete $\delta$-Bose gas,
which may be adapted to this setting by setting the inverse temperature 
parameter $\beta=\sqrt{-1}$.

\subsection{A simulation}

A simulation of the Markov chain with generator $G^\l$ on reverse plane partitions of 
square shape $\l=50^{50}$, started with $\pi_{ij}=50$ for all $(i,j)\in\l$ and run until
the first time that $\pi_{50,1}=0$, is shown below.  The height
of this profile over the box with coordinates $(i,j)$ is the value of $\pi_{ij}$
at this stopping time.

\medskip
\begin{center}
\includegraphics[scale=.54]{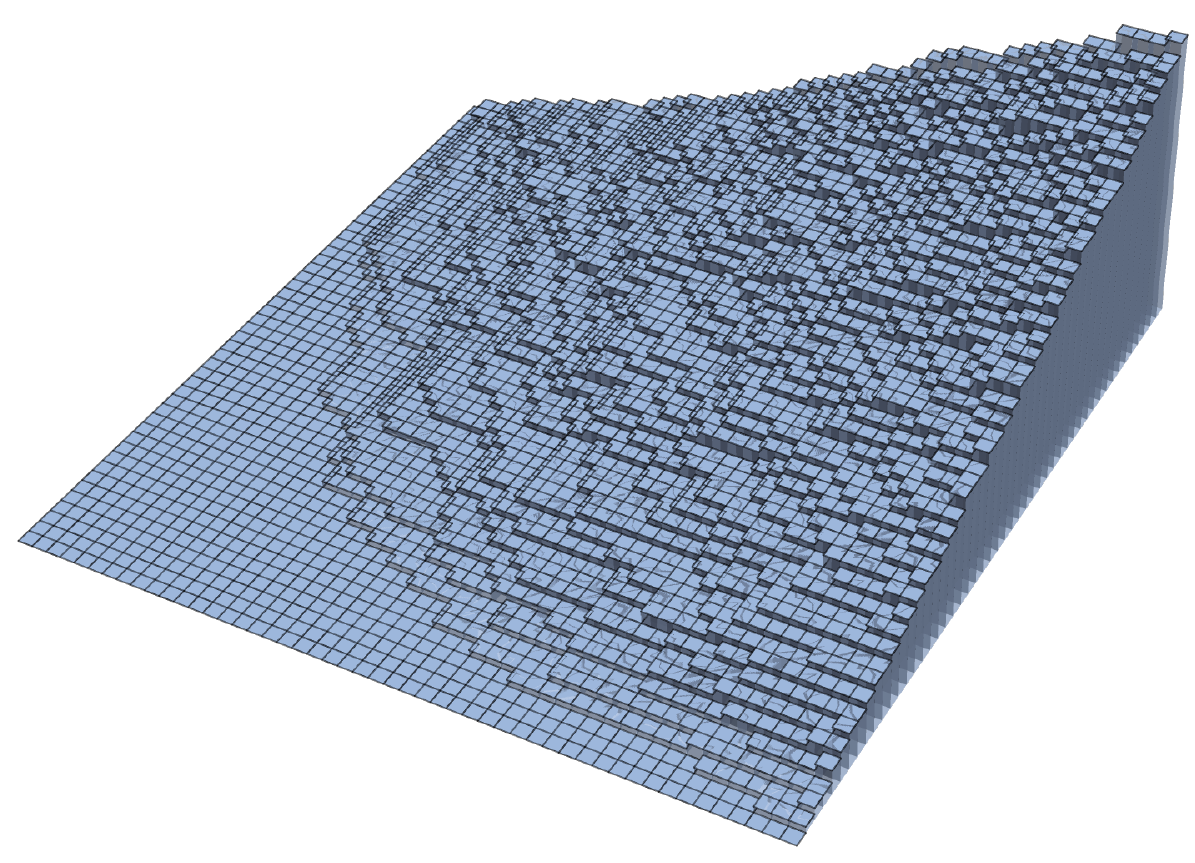}
\end{center}


\subsection{Outline of the paper}

In the next section we present the main results in the staircase setting.  
In Section~\ref{ia} we extend the range of parameters which may be considered
and describe some invariance properties of the associated Markov chains.
In Section \ref{tp} we discuss the Markov chain with generator $L^r$, with some explicit calculations for $r=1,2$,
and describe in detail the connections with imaginary exponential functionals of Brownian motion.
In Section \ref{rpp} we state and prove the main results 
in the context of general shapes and in the final section we outline some extensions to 
other root systems.

\section{Main results for staircase shapes}\label{mrss}

Denote by $L_k$ and $R_k$ the shift operators defined, for functions $f$ on $\Z_+$, by 
$$(L_k f)(k)=\begin{cases} f(k-1) & k>0\\0 & k=0\end{cases}$$
and
$$(R_k f)(k)=f(k+1).$$
The difference equation \eqref{rec} may be written as $h^{r} a_{r}=0$, where
\be\label{hr}
h^{r} = \sum_{i=1}^r L_{n_i} - \sum_{i=1}^{r} n_i^2 + \sum_{i=1}^{r-1}n_i n_{i+1} - \sum_{i=1}^{r} \alpha_i  n_i .
\ee
For $q_r$ defined by \eqref{L} and functions $f$ on $\Z_+^{r-1}$ define $q_r f$ on $\Z_+^r$ by
$$(q_r f)(n) = \sum_{k} q_r(n,k) f(k),$$
where the sum is over $0\le k_i\le n_i$, $i=1,\ldots,r-1$.

\begin{prop}\label{iq} The following intertwining relation holds:
\be\label{ir-f}
h^{r} \circ q_{r} = q_{r} \circ h^{r-1}.
\ee
\end{prop}
\begin{proof} Let
\be\label{Pr}
P_{r}(n)=\sum_{i=1}^{r} n_i^2 - \sum_{i=1}^{r-1}n_i n_{i+1} + \sum_{i=1}^{r} \alpha_i  n_i .
\ee
It suffices to show that
$$h^r_n q_r(n,k)=~^*\! h^{r-1}_k q_r(n,k),$$
where
$$~^*\! h^{r} = \sum_{i=1}^r R_{n_i} - P_{r}(n).$$

With the convention $k_0=k_r=0$, we compute
$$L_{n_i} q_r(n,k) = (n_i-k_i)(n_i-k_{i-1}+\alpha_{ir}) q_r(n,k),$$
hence
$$h^r_n q_r(n,k) = S_{n,k} \, q_r(n,k),$$
where
\bea
S_{n,k} &=& \sum_{i=1}^r (n_i-k_i)(n_i-k_{i-1}+\alpha_{ir}) - P_{r}(n)\\
&=& \sum_{i=1}^r (k_i k_{i-1} - k_in_i-n_ik_{i-1}+\alpha_{i+1,r} n_i - \alpha_{ir} k_i) +\sum_{i=1}^{r-1} n_i n_{i+1} .
\eea
Similarly,
$$R_{k_j} q_r(n,k) = (n_j-k_j)(n_{j+1}-k_j+\alpha_{j+1,r}) q_r(n,k),$$
hence
$$~^*\! h^{r-1}_k q_r(n,k) = T_{n,k}\, q_r(n,k),$$
where
\bea
T _{n,k} &=& \sum_{j=1}^{r-1} (n_j-k_j)(n_{j+1}-k_j+\alpha_{j+1,r}) - P_{r-1}(k)\\
&=& \sum_{j=1}^{r-1} (n_j n_{j+1} -k_j n_{j+1} - n_j k_j + \alpha_{j+1,r} n_j - \alpha_{jr} k_j) + \sum_{j=1}^{r-2} k_j k_{j+1}.
\eea
Thus $S_{n,k}=T_{n,k}$, as required.
\end{proof}

In the following we will assume that $(\alpha_1,\ldots,\alpha_r)\in\Z_+^r$ and denote
$$\beta_{ij}=\alpha_{i,i+j-1}=\alpha_i+\alpha_{i+1}+\cdots+\alpha_{i+j-1}.$$
Let $D=L-I$ denote the backward difference operator 
$$D_k f(k)=\begin{cases} f(k-1)-f(k) & k>0\\ -f(0) & k=0.\end{cases}$$
Since $h^r a_r=0$ and $a_r(n)>0$ for $n\in\Z_+^r$, the corresponding Doob transform
$$L^{r}=a_{r}(n)^{-1} \circ h^{r} \circ a_{r}(n) = \sum_{i=1}^r \frac{a_r(n-e_i)}{a_r(n)} D_{n_i}$$
generates a Markov chain on $\Z_+^r$.

If $r=1$ then, writing $n=n_1$ and $a=\alpha_1$, $L^1=n(n+a)D_n$.
If $r=2$ then, writing $n=\pi_{12}$, $m=\pi_{21}$, $k=\pi_{11}$, $a=\alpha_1$, $b=\alpha_2$,
and using the formula \eqref{bf}, 
$$L^2= \frac{n(n+a)(n+a+b)}{n+m+a+b} D_n+ \frac{m(m+b)(m+a+b)}{n+m+a+b} D_m .$$

Let $\Pi^{r}$ denote the set of non-negative integer arrays $(\pi_{ij},\ 1\le i+j \le r+1)$ satisfying 
$$\pi_{ij}\ge \pi_{i,j-1} \vee (\pi_{i-1,j}-\beta_{ij}),\qquad 1\le i+j\le r+1$$ 
with the convention $\pi_{i,0}=\pi_{0,j}=0$.  Note that if $\alpha_i=0$ for $1\le i \le r$
then $\Pi^r$ is the set of reverse plane partitions with staircase shape $\delta_{r+1}=(r,r-1,\ldots,1)$.
For $n\in\Z_+^r$, let $\Pi^r_n$ be the set of $\pi\in \Pi^r$ with $\pi_{i,r-i+1}=n_i$, $1\le i\le r$.  

For $\pi\in \Pi^{r}$, set
$$w_{r}(\pi)=\prod_{1\le i+j -1\le r} \frac1{(\pi_{ij}-\pi_{i,j-1})! (\pi_{ij}-\pi_{i-1,j}+\beta_{ij})!} .$$
Note that if $r=1$ and $\pi_{11}=n$, then $w_1(\pi)=a_1(n)$ defined by \eqref{def-a1}.
For $r\ge 2$, if $\pi\in\Pi^r_n$ and $\pi |_{\delta_r}\in\Pi^{r-1}_k$, where $n\in\Z_+^r$ and 
$k\in\Z_+^{r-1}$, then
\bea
\lefteqn{
\prod_{ i+j -1 = r} \frac1{(\pi_{i,j}-\pi_{i,j-1})! (\pi_{i,j}-\pi_{i-1,j}+\beta_{ij})!} }\\
&=& \prod_{i=1}^r \frac1{(\pi_{i,r-i+1}-\pi_{i,r-i})! (\pi_{i,r-i+1}-\pi_{i-1,r-i+1}+\beta_{i,r-i+1})!} \\
&=& \prod_{i=1}^r \frac1{(n_i-k_i)! (n_i-k_{i-1}+\alpha_{ir})! }
\eea
and hence
$$w_r(\pi)=q_r(n,k) w_{r-1}\left( \pi |_{\delta_r} \right) .$$
Moreover, for $n\in\Z_+^r$ and $k\in\Z_+^{r-1}$ we have $q_r(n,k)=0$ unless
$$n_i \ge k_i \vee (k_{i-1}-\alpha_{ir}),\qquad 1\le i \le r.$$  
It therefore follows from \eqref{def-ar} that we can write
$$a_{r}(n)=\sum_{\pi\in \Pi^r_n} w_{r}(\pi).$$

For $n\in\Z_+^r$, let $K^{r}_n$ be the probability distribution on $\Pi^r_n$
defined by $$K^{r}_n(\pi)=w_{r}(\pi) / a_{r}(n).$$
For $\pi\in\Pi^r$ and $1\le i+j\le r+1$, set
$$b_{ij}(\pi)=(\pi_{ij}-\pi_{i,j-1})(\pi_{ij}-\pi_{i-1,j}+\beta_{ij}),$$
with the convention $\pi_{i,0}=\pi_{0,j}=0$.
Let
\be\label{gr}
G^{r} = \sum_{1\le i+j\le r+1} b_{ij}(\pi) D_{\pi_{ij}}.
\ee
For example, if $r=2$ then, writing $n=\pi_{12}$, $m=\pi_{21}$, $k=\pi_{11}$, $a=\alpha_1$, $b=\alpha_2$,
$$G^2 = k(k+a) D_k + n(n-k+a+b) D_n + m(m-k+b) D_m.$$

\begin{thm}\label{mf-nu}
Let $\pi(t),\ t\ge 0$ be a Markov chain on $\Pi^r$ with generator $G^{r}$
and initial law~$K^{r}_n$, for some $n\in \Z_+^{r}$.
Then $N(t)=(\pi_{1,r}(t),\ldots,\pi_{r,1}(t))$ is a Markov chain on $\Z_+^{r}$ with generator $L^{r}$
and, for all $t>0$, the conditional law of $\pi(t)$ given $\{N(s),\ s\le t\}$ is $K^{r}_{N(t)}$.
\end{thm}
\begin{proof}
For functions $f$ on $\Z_+^r\times\Z_+^{r-1}$, define $\tilde q_r f$ on $\Z_+^r$ by
$$(\tilde q_r f) (n) = \sum_k q_r(n,k) f(n,k),$$
where the sum is over $0\le k_i\le n_i$, $i=1,\ldots,r-1$.  
Let 
$$g^r = h^{r-1}_k + \sum_{i=1}^r \frac{L_{n_i} q_r(n,k)}{q_r(n,k)} D_{n_i} =
h^{r-1}_k + \sum_{i=1}^r (n_i-k_i)(n_i-k_{i-1}+\alpha_{ir}) D_{n_i} .$$
By Proposition~\ref{iq},
\bea
[h^r (\tilde q_r f)] (n) &=& \sum_k h^r_n [q_r(n,k) f(n,k)] \\
&=& \sum_k  \left[ (h^r_n q_r(n,k)) f(n,k) + \sum_{i=1}^r L_{n_i} q_r(n,k) D_{n_i} f(n,k) \right] \\
&=& \sum_k  \left[  q_r(n,k) h^{r-1}_k f(n,k) + \sum_{i=1}^r L_{n_i} q_r(n,k) D_{n_i} f(n,k) \right] \\
&=& \sum_k q_r(n,k) g^r f(n,k) \\
&=& [\tilde q_r (g^r f)](n).
\eea
The statement of theorem follows, by induction and the theory of Markov functions~\cite{kurtz,rp}.
The application of the latter in this context (and indeed for all the examples considered in this paper)
is free of technical considerations since,
given the initial law, the relevant part of the state space is finite.
\end{proof}

In Section~\ref{rpp} we will present a more general version of Theorem~\ref{mf-nu},
valid for arbitrary shapes.  We will also prove, again in a more general setting, the existence 
of a unique entrance law for the Markov chain with generator $G^r$,
starting from the array with all entries equal to $+\infty$.  The following is a special case of 
Theorem~\ref{el}.

\begin{thm}\label{el-nu} The Markov chain on $\Pi^r$ with generator $G^r$ has a unique
entrance law starting from $\pi_{ij}=+\infty$ for $1\le i+j\le r+1$.  Under this entrance law, 
$N(t)=(\pi_{1,r}(t),\ldots,\pi_{r,1}(t))$ is a Markov chain on $\Z_+^{r}$ with generator $L^{r}$
and, for all $t>0$, the conditional law of $\pi(t)$ given $\{N(s),\ s\le t\}$ is $K^{r}_{N(t)}$.
\end{thm}

\begin{rem}
When $\alpha\equiv 0$, the normalised coefficients 
$$A_r(n)=\left(\prod_{i=1}^r n_i!^2\right) a_r(n)$$
are given, for $n\in\Z_+^r$, by the binomial sum formula
$$A_r(n)=\sum_{\pi\in\Pi^r_n} W_r(\pi),\qquad
W_r(\pi)=\prod_{(i,j)\in\delta_r} {\pi_{i,j+1} \choose \pi_{ij}} {\pi_{i+1,j} \choose \pi_{ij}}.$$
These are the unique solution to $H^r A_r=0$ on $\Z_+^r$ with $A_r(0)=1$, where
$$H^r = \sum_{i=1}^r n_i^2 D_{n_i} + \sum_{i=1}^{r-1} n_i n_{i+1}.$$
\end{rem}
\begin{rem}
The diagonal values $a_n=A_3(n,n,n)$ are the Ap\'ery numbers
$$a_n=\sum_k {n \choose k}^2 {n +k \choose k}^2.$$
associated with $\zeta(3)$.  This sequence satisfies the three-term recurrence
\be\label{rec-Ap}
n^3 a_n = (34n^3-51n^2+27n-5)a_{n-1}-(n-1)^3 a_{n-2},
\ee
with $a_0=1$ and $a_1=5$.  We remark that this recurrence may
be derived, in an elementary way, using the connection to the
Toda lattice.  The latter implies that $A_3(n,m,l)$ is annihilated 
by the three commuting difference operators
$$n^2 D_n + m^2 D_m +l^2 D_l+nm+ml,\quad mn^2D_n+(n-l)m^2D_m+ml^2D_l,$$
$$n^2l^2D_{n,l}-l(l-m)n^2D_n-nm^2lD_m+n(m-n)l^2D_l,$$
where $D_{n,l}f(n,l)=f(n-1,l-1)-f(n,l).$
The corresponding difference equations may be combined to obtain \eqref{rec-Ap}.
\end{rem}

\section{More general parameters and symmetries}\label{ia}
The assumption $(\alpha_1,\ldots,\alpha_r)\in\Z_+^r$ is not necessary for
some version of Theorem~\ref{mf-nu} to hold. 
For example, if $(\alpha_1,\ldots,\alpha_r)\in\Z^r$ then the statement may be
modified as follows.  
Denote by $\Pi^{r,\alpha}$ the set of non-negative integer arrays $(\pi_{ij},\ 1\le i+j \le r+1)$ satisfying 
$$\pi_{ij}\ge \omega_{ij} \vee \pi_{i,j-1} \vee (\pi_{i-1,j}-\beta_{ij}),\qquad 1\le i+j\le r+1,$$ 
where $\omega=(\omega_{ij},\ 1\le i+j \le r+1)$ is the unique solution to 
\be\label{w-rec}
\omega_{ij}=\omega_{i,j-1}\vee (\omega_{i-1,j}-\beta_{ij}),\qquad 1\le i+j\le r+1,
\ee
with the conventions $\pi_{i,0}=\pi_{0,j}=\omega_{i,0}=\omega_{0,j}=0$.
This agrees with the previous definition when $(\alpha_1,\ldots,\alpha_r)\in\Z_+^r$,
since $\omega\equiv 0$ in that case.
Similarly, let $\Z_+^{r,\alpha}$ denote the set of $n\in\Z_+^r$ satisfying
$n_i\ge\omega_{i,r-i+1}$, $i=1,\ldots,r$.  Then the statement and proof of Theorem~\ref{mf-nu}
remain valid, as written, with $\Pi^r$ and $\Z_+^r$ replaced by $\Pi^{r,\alpha}$ and $\Z_+^{r,\alpha}$,
respectively.

The array $\omega\in\Pi^{r,\alpha}$ is an absorbing state
for the Markov chain with generator $G^r$.  
It is given explicitly as follows.  Consider the triangular array
$(\nu_i^k,\ 1\le i\le k\le r+1)$ defined by $\alpha_i=\nu^k_i-\nu^k_{i+1}$
for $1\le i\le k\le r+1$ and $\nu_1^k+\nu_2^k+\cdots+\nu_k^k=0$ for $1\le k\le r+1$.
For each $k$, denote by $\tilde\nu^k_1\ge\tilde\nu^k_2\ge\cdots\ge\tilde\nu^k_k$
the decreasing rearrangement of $\nu^k_1,\ldots,\nu^k_k$.  Then, for 
$1\le i\le k\le r+1$, we have
$$\omega_{i,k-i}=\tilde \nu^k_1+\cdots\tilde \nu^k_i - \nu^k_1-\cdots-\nu^k_i.$$
It is straightforward to verify that this satisfies \eqref{w-rec}, by induction over $k$.

Note that if $(\alpha_1,\ldots,\alpha_r)\in\Z_-^r$, then
$$\omega_{ij}=-\beta_{1j}-\beta_{2j}-\cdots-\beta_{ij},\qquad 1\le i+j\le r+1.$$
This leads us to observe the following basic symmetry.  
For $(\alpha_1,\ldots,\alpha_r)\in\Z^r$, define
$$\gamma_{ij}=-\beta_{1j}-\beta_{2j}-\cdots-\beta_{ij},\qquad 1\le i+j\le r+1.$$
From the definitions, if $\pi\in\Pi^r=\Pi^{r,\alpha}$ then $\hat\pi\in\Pi^{r,-\alpha}$, where
$$\hat\pi_{ij}=\pi_{ji}-\gamma_{ji}, \qquad 1\le i+j\le r+1.$$
Moreover, if $\pi(t)$ is a Markov chain on $\Pi^{r,\alpha}$ with generator $G^r=G^{r,\alpha}$ 
then $\hat\pi(t)$ is a Markov chain on $\Pi^{r,-\alpha}$ with generator $G^{r,-\alpha}$.
If $(\alpha_1,\ldots,\alpha_r)\in\Z_-^r$ then $\omega=\gamma$, and the transformation $\pi\mapsto \hat\pi$ 
takes us back to the simplest case with $\omega\equiv0$.

Now consider the Markov chain with state space $\Z^{r,\alpha}$ and generator 
$L^r=L^{r,\alpha}$, where $\alpha\in\Z^r$.  Let us write $\alpha_i=\nu_i-\nu_{i+1}$, 
$i=1,\ldots,r$, where $\nu\in (\Z/(r+1))^{r+1}$ with $\sum_i\nu_i=0$, and
denote $a_r(n)=a_{r,\nu}(n)$.
From the above,
$$\omega_{i,r-i+1}=\tilde \nu_1+\cdots\tilde \nu_i - \nu_1-\cdots-\nu_i,\qquad i=1,\ldots,r$$
where $\tilde\nu_1\ge\tilde\nu_2\ge\cdots\ge\tilde\nu_{r+1}$
is the decreasing rearrangement of $\nu_1,\ldots,\nu_{r+1}$.

For $n\in\Z_+^{r,\alpha}$, define 
$$n'_i=n_i+ \nu_1+\cdots+\nu_i, \qquad j=1,\ldots,r.$$  
Let 
$$S^{r,\nu}=(\Z_+^{r,\alpha})'=(\tilde\nu_1,\tilde\nu_1+\tilde\nu_2,\ldots,\tilde\nu_1+\cdots+\tilde\nu_r)+\Z_+^r,$$ 
and note that this set is invariant under permutations 
of the parameters $\nu_1,\ldots,\nu_{r+1}$.
For $n\in\Z_+^{r,\alpha}$, define $\tilde a_{r,\nu}(n')=a_{r,\nu}(n)$.
For example,
$$\tilde a_{1,\nu}(n')=\frac1{(n'-\nu_1)! (n'-\nu_2)!}$$
and, using \eqref{bf},
$$\tilde a_{2,\nu}(n',m')=(n'+m')!\prod_{i=1}^3\frac1{(n'-\nu_i)!(m'+\nu_i)!}.$$
\begin{prop}
For any fixed $n'\in S^{r,\nu} $, $\tilde a_{r,\nu}(n')$ is invariant under 
permutations of the parameters $\nu_1,\ldots,\nu_{r+1}$.
\end{prop}
\begin{proof}
This follows from a more general form of the recursion \eqref{def-ar} 
given in \cite[Theorem 15]{is}.  Let $t\in\{1,2,\ldots,r+1\}$
and define, for $n\in\Z_+^r$ and $k\in\Z_+^{r-1}$,
$$q_{r,\nu}(n,k)=\prod_{i=1}^{t-1} \frac1{(n_i-k_i)!(n_i-k_{i-1}+\nu_i-\nu_t)!}
\prod_{i=t}^{r} \frac1{(n_i-k_{i-1})!(n_i-k_{i}-\nu_{i+1}+\nu_t)!}.$$
Then
$$a_{r,\nu}(n)= \sum_{k\in\Z_+^{r-1}} q_{r,\nu}(n,k) a_{r-1,\mu}(k),$$
where
$$\mu=\left( \nu_1+\nu_t/r,\ldots,\nu_{t-1}+\nu_t/r,\nu_{t+1}+\nu_t/r,\ldots,\nu_{r+1}+\nu_t/r\right).$$
Changing variables to
$$n'_i=n_i+\nu_1+\cdots+\nu_i,\quad i=1,\ldots,r+1,$$
$$k'_i=k_i+\mu_1+\cdots+\mu_i,\quad i=1,\ldots,r,$$
we can write $q_{r,\nu}(n,k)=\tilde q_{r,\nu_t}(n',k')$, where
$$\tilde q_{r,\theta}(n',k')=\prod_{i=1}^r \frac1{(n'_i-k'_i+j\theta/r)!(n'_i-k'_{i-1}-(r-i+1)\theta/r)!}.$$
In this notation,
$$\tilde a_{r,\nu}(n') = \sum_{k'} \tilde q_{r,\nu_t}(n',k') \tilde a_{r-1,\mu}(k'),$$
where the sum is over $k'$ such that $k\in\Z_+^{r-1}$.
Since $t$ is arbitrary, the claim follows by induction over $r$.
\end{proof}
\begin{cor}
If $N$ is a Markov chain in $\Z_+^{r,\alpha}$ with generator $L^{r,\alpha}$
then the law of the Markov chain $N'$, which has state space $S^{r,\nu}$, 
is invariant under permutations of the parameters $\nu_1,\ldots,\nu_{r+1}$.
\end{cor}

\section{Transition probabilities and hitting times}\label{tp}

Let $(\alpha_1,\ldots,\alpha_r)\in\Z_+^r$ and consider the Markov chain $(X(t),\ t\ge 0)$ on $\Z_+^r$
with generator $L^{r}$.  Denote by $\P_n$ the law of this process started from $n\in\Z_+^r$ and $\E_n$
the corresponding expectation.  By Theorem~\ref{el-nu}, this law is also well defined for $n=+\infty$,
in this case we will write $\P=\P_{+\infty}$ and $\E=\E_{+\infty}$.  
For $n\in\Z_+^r$, let $$T_n=\inf\{ t\ge 0:\ X(t)=n\}$$ denote the first hitting time of $n$,
with the usual convention $\inf\emptyset=+\infty$.

The transition probabilities are given by
$$p^r_t(n,m) = \frac{a_{r}(m)}{a_{r}(n)} \tilde p^{r}_t(n,m) ,$$
where $\tilde p^{r}_t(n,m)$ is the heat kernel associated with $h^{r}$.
Let
$$g_{r}(n,m) = \int_0^\infty p^{r}_t(n,m) dt,$$
and note that, since the rate at which the chain leaves a state $m\in \Z_+^r$ is $P_{r}(m)$, 
$$\P_{n} (T_m<\infty) = P_{r}(m) g_{r}(n,m).$$

By Theorem~\ref{el-nu}, we may define
$$p^{r}_t(m) = \lim_{n\to\infty} p^{r}_t(n,m),\qquad h^{r}_m = \P(T_m<\infty)$$
and, for $m\in\Z_+^r\backslash\{0\}$, 
$$g_{r}(m) = \lim_{n\to\infty} g_{r}(n,m),\qquad a_{r}^*(m) = g_r(m) / a_{r}(m).$$
Note that $h^r_0=1$ and, for $n\in\Z_+^r\backslash\{0\}$,
\be\label{hr2} h^{r}_n = P_{r}(n) a_{r}(n) a_{r}^*(n).
\ee

\begin{prop}\label{fc}
For $r\ge 2$ and $k\in\Z_+^{r-1}\backslash\{0\}$,
$$a_{r-1}^*(k) = \sum_{n} q_r(n,k) a_r^*(n),$$
where the sum is over $n\in\Z_+^r$ such that $n_i\ge k_i$ for $i=1,\ldots,r$.
\end{prop}
\begin{proof} By Theorem~\ref{el-nu}, 
$$p^{r-1}_t(k) =  \sum_{n} a_r(n)^{-1} q_r(n,k) a_{r-1}(k) p^r_t(n).$$
Integrating over $t>0$ gives, by monotone convergence, 
$$g_{r-1}(k) =  \sum_{n} a_r(n)^{-1} q_r(n,k) a_{r-1}(k) g_r(n).$$
The statement of the proposition follows.
\end{proof}

Throughout the remainder of this section 
we adopt the convention that empty products are 
equal to one and empty sums are equal to zero.

\subsection{The case $r=1$}

Writing $n=n_1$ and $a=\alpha_1$, $L^1=n(n+a)D_n$ and
$$a_1(n)=\frac1{n!(n+a)!},\qquad g_1(n)=\frac1{n(n+a)}, \qquad a_1^*(n)=\G(n)\G(n+a).$$ 
When $a=1$ this is essentially Kingman's coalescent~\cite{kingman}.
\begin{prop}\label{a1-kmg} 
For $k\ge n$, $\Re(s)> -1-a$ and $t\ge 0$,
$$\E_{k} e^{-s T_n} = \prod_{j=n+1}^{k} \frac{j(j+a)}{j(j+a)+s}$$
and
$$\tilde p_t(k,n)=\sum_{j=n}^{k} \prod_{l=n, l\ne j}^{k} \frac1{(l-j)(l+j+a)} e^{-j(j+a) t}.$$
\end{prop}
\begin{proof}
The first claim is immediate from the fact that, under $\P_k$, $T_n$ is distributed
as a sum of independent exponential random variables with respective parameters
$j(j+a)$, $j=n+1,\ldots,k$.  

For the second claim, define
$$\varphi_j(n)=\frac1{(n-j)!(n+a+j)!},\qquad \varphi_j^*(n)=(-1)^j \left[ (1+j)_a+(j)_a\right] (-j)_n (a+j)_n,$$
and note that 
$$h^1 \varphi_j(n) = -j(j+a) \varphi_j(n),\qquad ~^*\! h^1 \varphi_j^*(n)=-j(j+a) \varphi_j^*(n),$$
where $h^1=L_n-n(n+a)$ and $~^*\! h^1=R_n-n(n+a)$ as defined previously.
The expression given for $\tilde p_t(k,n)$ in the statement of the proposition is equivalent to
$$\tilde p_t(k,n)=\sum_{j=n}^{k} \varphi_j(k) \varphi_j^*(n) e^{-j(j+a) t}.$$
This satisfies the required forward and backward equations
$$\partial_t \tilde p_t(k,n) = h^1_k \tilde p_t(k,n) = ~^*\! h^1_n \tilde p_t(k,n),$$
so it only remains to show that $f(k,n)=\delta_{kn}$, where
$$f(k,n)=\sum_{j=n}^{k} \prod_{l=n, l\ne j}^{k} \frac1{(l-j)(l+j+a)} .$$
Clearly $f(n,n)=1$.  Writing $(l-j)(l+j+a)=l(l+a)-j(j+a)$, the fact that
$f(k,n)=0$ for $k>n$ may be seen as a consequence of the more general identity
(see for example~\cite[Eq. 8.17]{noc13})
\be\label{gi}
\sum_{a=1}^N \prod_{b=1, b\ne a}^N \frac1{\l_a-\l_b} = 0.
\ee
\end{proof}

\begin{rem}
The functions $\tilde p_t(k,n)$ are also related to the classical Toda chain, 
see for example~\cite[Theorem 8.5]{noc13}.  In particular, they satisfy the Toda equations
$$\tilde p_t(k,n)\partial_t^2 \tilde p_t(k,n)-(\partial_t \tilde p_t(k,n))^2=\tilde p_t(k-1,n+1)\tilde p_t(k,n)-\tilde p_t(k-1,n)\tilde p_t(k,n+1).$$
\end{rem}

\subsection{The case $r=2$ and $\alpha\equiv 0$}

In this case, writing $(n,m)=(n_1,n_2)$,
$$a_{2}(n,m)=\frac{(n+m)!}{n!^3m!^3} ,\qquad 
L^2= \frac{n^3}{n+m} D_n+ \frac{m^3}{n+m} D_m .$$

We first note the following decomposition for the absorption time.
For the Markov chain with generator $L^2$,
let $\tau_{n,m}$ be the first hitting time of $(0,0)$ for the chain started at $(n,m)$.  
For the Markov chain with generator $n^2 D_n$,
let $\tau_n$ be the first hitting time of $0$ for the chain started at $n$.
\begin{prop}\label{dd}
The absorption time $\tau_{n,m}$ has the same law as $\tau_n + \tau'_m$,
where $\tau'_m$ is an independent copy of $\tau_m$.
This identity in law remains valid for $n=m=+\infty$.
\end{prop}
\begin{proof}
It suffices to show that $H_{s}(n,m)=G_{s}(n)  G_{s}(m) $, where
$$H_{s}(n,m)=\E e^{-s \tau_{n,m}} ,\qquad G_{s}(n) = \E e^{-s \tau_n}.$$
The function $H_{s}(n,m)$ is the unique solution to $L^2H=sH$ on $\Z_+^2\backslash\{(0,0)\}$ with $H(0,0)=1$,
and $G_s(n)$ satisfies $n^2 D_n G_s(n)=s G_s(n)$ for $n>0$.
Let $\tilde H_s(n,m)=G_{s}(n)  G_{s}(m)$.
Clearly $\tilde H_s(0,0)=1$.  
For $n>0$ and $m=0$ we have $L^2=m^2 D_m$, hence $L^2 \tilde H_s=s \tilde H_s$ on this set.
Similarly, $L^2 \tilde H_s=s \tilde H_s$ for $n=0$ and $m>0$.
For $n,m>0$, we have
\bea
L^2 \tilde H_s(n,m) &=& \frac{n^3}{n+m} D_n [G_{s}(n)  G_{s}(m)] + \frac{m^3}{n+m} D_m [G_{s}(n)  G_{s}(m)] \\
&=& s \frac{n}{n+m} [G_{s}(n)  G_{s}(m)] + s \frac{m}{n+m} [G_{s}(n)  G_{s}(m)] 
= s \tilde H_s(n,m),
\eea
hence $\tilde H_s=H_s$, as required.
\end{proof}

\begin{rem} The random variables $S_1=2\tau_{\infty}/\pi^2$ and 
$S_2=2(\tau_\infty + \tau'_\infty)/\pi^2$
have many interesting interpretations and applications~\cite{bpy,py}.
\end{rem}

\begin{rem} A key ingredient in the proof of Proposition~\ref{dd} is the fact that
$$A_2(n,m) = n!^2 m!^2 a_2(n,m) = {n+m \choose n}$$
satisfies the Pascal relation
$$A_2(n,m)=A_2(n-1,m)+A_2(n,m-1).$$
In general, for $\alpha\equiv 0$ and $r>2$, the normalised coefficients 
$$A_r(n)=\left(\prod_{i=1}^r n_i!^2\right) a_r(n)$$
do not satisfy
$$A_r(n)=\sum_{i=1}^r A(n-e_i),$$
so the above factorisation property does not extend in the obvious way to $r>2$.
\end{rem}

\begin{cor}\label{cauchy}
$$\sum_{(n,m)\ne (0,0)} a_2(n,m) a_2^*(n,m) = 2 \zeta(2),\quad \sum_{(n,m)\ne (0,0)} \frac1{n!^2m!^2} a_2^*(n,m) = \zeta(2).$$
\end{cor}
\begin{proof}
The expected time spent at $(n,m)\ne(0,0)$ is given by $a_2(n,m) a_2^*(n,m)$,
so the first identity follows from Proposition~\ref{dd}.
By Proposition~\ref{fc},
\be\label{1k2}
\sum_{n,m\ge k} q_2((n,m),k) a_1(k) a_2^*(n,m) = \frac1{k^2}. 
\ee
Summing over $k\ge 1$ gives
$$\sum_{n,m\ge 1}\left[a_2(n,m) - q_2((n,m),0) \right] a_2^*(n,m) = \zeta(2)$$
or, equivalently,
$$\sum_{(n,m)\ne (0,0)} \left[a_2(n,m) - q_2((n,m),0) \right] a_2^*(n,m) = \zeta(2).$$
Combined with the first identity this gives the second.
\end{proof}

\begin{rem} Proposition~\ref{dd} may be extended to the case $(\alpha_1,\alpha_2)\in\Z^2$
with $\alpha_1+\alpha_2=0$ as follows.  Suppose
$a=\alpha_1\ge 0$ and let $E=\{(n,m)\in\Z_+^2:\ m\ge a\}$.  Then, for $(n,m)\in E$,
$$a_{2}(n,m)=\frac{(n+m)!}{n!^2(n+a)!m!^2(m-a)!},\quad L^2= \frac{n^2(n+a)}{n+m} D_n+ \frac{m^2(m-a)}{n+m} D_m .$$
Let $\tau_{n,m}$ be the first hitting time of $(0,a)$ for the Markov chain with generator $L^2$ started at $(n,m)\in E$.
Let $\tau_n$ be the first hitting time of $0$ for the Markov chain with generator $n(n+a)D_n$ started at $n\ge 0$,
and note that $\tau_{m-a}$ has the same law as the first hitting time of $a$ for the Markov chain with
generator $m(m-a)D_m$ started at $m\ge a$.  Then it holds that $\tau_{n,m}$ has the same law as 
$\tau_n+\tau_{m-a}'$, where $\tau_{m-a}'$ is an independent copy of $\tau_{m-a}$. This identity
remains valid for $n=m=+\infty$, and the first identity of Corollary~\ref{cauchy} extends to
$$\sum_{(n,m)\in E\backslash\{(0,a)\}} a_2(n,m) a_2^*(n,m) = \sum_{k=1}^\infty \frac2{k(k+a)}.$$
\end{rem}

\begin{prop}\label{prop-hp}
For $n,m\ge 1$,
$$\P_{(k,l)}(  T_{n,m}<\infty) = \frac{n^3+m^3}{n^3 m^3}  
\sum_{j=n}^k j^3 \prod_{a=n, a\ne j}^k \frac{a^3}{a^3-j^3} \prod_{b=m}^l \frac{b^3}{b^3+j^3}.$$
For $n\ge 1$,
$$\P_{(k,l)}(  T_{n,0}<\infty) = \sum_{j=n}^k 
 \prod_{a=n, a\ne j}^k \frac{a^3}{a^3-j^3} \prod_{b=1}^l \frac{b^3}{b^3+j^3}.$$
\end{prop}
\begin{proof}
Let
$$~^*\! L^2 = B_n\circ \frac{n^3}{n+m}  + B_m \circ \frac{m^3}{n+m},$$
where $B_k$ denotes the forward difference operator $B_kf(k)=f(k+1)-f(k)$.

For $\nu\in\C^3$ with $\sum_i\nu_i=0$, the function
$$\psi_\nu(n,m)=\frac{n!^3 m!^3}{\prod_{i=1}^3 (n-\nu_i)! (m+\nu_i)!}$$
is an eigenfunction of $L^2$ with eigenvalue $-\sum_i\nu_i^2/2$, and
$$\psi^*_\nu(n,m)=\frac{n+m}{n!^3 m!^3} \prod_{i=1}^3 (-\nu_i)_n (\nu_i)_m $$
is an eigenfunction of $~^*\! L^2$ with the same eigenvalue.  
These claims, which may be anticipated from the formula~\eqref{bf}, are easily verified.

Let $\nu=(1,\omega , \omega^2 )$, where $\omega$ is the 
primitive cube root of unity $\omega=-1/2+i\sqrt{3}/2$.
Note that $\sum_i\nu_i=\sum_i \nu_i^2=0$, and
$$x^3+j^3=(x+j)(x+j\omega )(x+j\omega^2 ).$$
Define $f^*_0(n,m)=\delta_{n,0}$ and, for $j\ge 1$, 
$$f^*_j(n,m)=\psi^*_{j\nu}(n,m)= \frac{n+m}{n!^3 m!^3} 
\prod_{a=0}^{n-1} (a^3-j^3) \prod_{b=0}^{m-1} (b^3+j^3).$$
Define $f_0(k,l)=1$ and, for $j\ge 1$,
$$f_j(k,l)=k!^3 l!^3 \prod_{a=0, a\ne j}^k \frac1{a^3-j^3} \prod_{b=0}^l \frac1{b^3+j^3}. $$
Then $L^2f_j=~^*\! L^2 f_j^*=0$, for all $j\ge 0$.
We will show that
$$\P_{(k,l)}(  T_{n,m}<\infty) = \sum_{j=n}^k f_j(k,l) f_j^*(n,m),$$
which implies the statement of the proposition.
The function 
$$F(n,m)=\sum_{j=n}^k f_j(k,l) f_j^*(n,m)$$
satisfies $~^*\! L^2 F=0$ and, for $(k,l)\ne (0,0)$,
$$ F(k,l)=\frac1{k^2+l^2-kl}.$$
It therefore suffices to show that
$$\sum_{j=n}^{k} f_j(k,l) f^*_j(n,l+1)=0,\qquad 0\le n<k,$$
or equivalently
$$\sum_{j=n}^k \prod_{a=n, a\ne j}^k \frac1{a^3-j^3} = 0,\qquad 0\le n<k.$$
This follows from \eqref{gi}.
\end{proof}

Let $R_n(x)=\prod_{a=1}^n (a^3+x^3)$ with $R_0(x)=1$,
and set $h_{nm}=\P(T_{n,m}<\infty)$.
\begin{cor}\label{hti}
For $n,m\ge 1$, 
\be\label{hnm}
h_{nm} = \frac{n^3+m^3}{n!^3 m!^3}   \sum_{j=n}^\infty R_{n-1}(-j) R_{m-1}(j) \frac{3\pi^2j^5}{\sin^2(\pi \omega j)}
\ee
and, for $n\ge 1$,
$$h_{n0}=h_{0n}=\frac1{\G(n)^3} \sum_{j=n}^\infty R_{n-1}(-j) \frac{3\pi^2j^2}{\sin^2(\pi\omega j)}.$$
\end{cor}
The above corollary also provides formulas for $a_2^*(n,m)$, using
\be\label{ha}
h_{nm}= (n^2+m^2-nm) a_2(n,m) a_2^*(n,m).
\ee
We note that the hitting probabilities satisfy
\be\label{ht-rel}
h_{nm}=h_{mn},\qquad h_{n,0}+h_{n-1,1}+\cdots+h_{0,n}=1.
\ee

In the above formulas, noting that $R_{n-1}(-j)=0$ for $j=1,\ldots,n-1$, the summations
may be extended to $j\ge 1$.  Thus, if we define, for $k\ge 0$,
$$S_k = \sum_{j=1}^\infty \frac{6\pi^2j^{2+3k}}{\sin^2(\pi\omega j)},$$
then each hitting probability can be expressed as
a finite rational linear combination of these series.
By Proposition 2.24 of \cite{b} we have $S_0=1$ and,
by (a corrected version of) Corollary 2.23 in that paper, 
$S_k=0$ for all positive, even values of $k$.
The identity $S_{2k}=\delta_{k0}$
may also be inferred from Corollary~\ref{hti}
together with the relations \eqref{ht-rel}.  For example, by Corollary~\ref{hti}
we have $h_{10}=S_0/2$; together with $h_{10}=h_{01}$ and
$h_{01}+h_{10}=1$, this implies $S_0=1$.

The first few hitting probabilities are given by 
$$h_{00}=1,\quad h_{10}=h_{01}=1/2,\quad h_{11}=S_1=0.87987\ldots,\quad h_{20}=(1- S_1)/2,$$ 
$$h_{21}=9 S_1/16,\quad h_{30}=(8-9S_1)/16, \quad h_{22}=(S_1-S_3)/8.$$ 

\begin{rem}
We can write $S_k = 24 \pi^2 T_{2+3k}$, where
$$T_r = - \sum_{n=1}^\infty \frac{n^r q^n}{(1-q^n)^2},\qquad q=e^{2\pi i\omega}=-e^{-\pi\sqrt{3}}.$$
These are well known $q$-series and may be evaluated for even values of $r$
using formulas due to Ramanujan~\cite{haw}.  For example,
$$T_2= \frac1{24\pi^2},\quad
T_4= \frac{3 \Gamma(\tfrac13)^{18}}{40960\pi^{12}},\quad
T_6 = \frac{3\sqrt{3}  \Gamma(\tfrac13)^{18}}{28672\pi^{13}},\quad
T_8= 0,\quad  \ldots$$
\end{rem}

\begin{rem}
The summation formula \eqref{1k2} may be verified directly using \eqref{hnm}, \eqref{ha} and~\cite[Equation (34)]{pvj}.
\end{rem}

\subsection{Imaginary exponential functionals of Brownian motion}\label{iefb}

Let $\alpha\equiv 0$ and denote by $p_t(n,m)$ the transition kernel of the Markov chain with generator 
$$L^r=a_r(n)^{-1}\circ h^r \circ a_r(n).$$
Let $B=(B_1,\ldots,B_{r+1})$ be a standard Brownian motion in $\R^{r+1}$,
started at the origin, and set
$$Y_k(t)=e^{i(B_k(t)-B_{k+1}(t))},\qquad Z_k(t)=\int_0^t Y_k(s) ds,\qquad k=1,\ldots, r.$$
We will use the following notation.  For $0\ne y,z\in\C^r$ and $n\in\Z_+^r$,
$$y^{-n}=\prod_{k=1}^r y_k^{-n_k},\quad z^n=\prod_{k=1}^r z_k^{n_k},\quad n!=\prod_{k=1}^r n_i! .$$
\begin{prop}\label{ief1} For $n,m\in\Z_+^r$ with $n\ge m$,
$$p_t(n,m)=\frac{a_r(m)}{a_r(n)}  \frac1{(n-m)!} \E  \left[Y(t)^{-n} Z(t)^{n-m}\right].$$
In particular,
$$\P_n(T_0\le t) = p_t(n,0) =  \frac{\E Z(t)^{n}}{a_r(n) n!}.$$
\end{prop}
\begin{proof}
It is straightforward to show that $h^r_n F(n,x) = \I^r_x F(n,x)$, where
$$\I^r=\frac12\sum_{j=1}^{r+1} \frac{\partial^2}{\partial x_j^2}+\sum_{k=1}^r y_k,$$
$$F(n,x)=y^{-n}, \qquad y_k=e^{i(x_k-x_{k+1})},\quad k=1,\ldots, r .$$
Setting $H(n,x)=a_r(n)^{-1} F(n,x)$, this implies the Markov duality relation 
$$L^r_n H(n,x) = \I^r_x H(n,x).$$ 
It follows, for example using \cite[Theorem 4.4.11]{ek}, that,
if $N(t)$ is a Markov chain with generator $L^r$ started at $n$ then
$$\E H(N(t),x) = \E \left[ H(n,x+B(t)) \exp\left(\sum_{k=1}^r y_k Z_k(t)\right) \right] .$$
Equivalently,
$$\sum_{0\le m\le n} p_t(n,m) a_r(m)^{-1}  y^{-m} = a_r(n)^{-1} y^{-n}
\E\left[ Y(t)^{-n} \exp\left(\sum_{k=1}^r y_k Z_k(t)\right)\right] .$$
Since $|Y_k(t)|=1$ and $|Z_k(t)|\le t$ almost surely, for each $k$,
the first statement of the proposition follows, using bounded convergence 
and Cauchy's theorem.  Since the law of Brownian motion is invariant under time reversal,
$Y(t)^{-1}Z(t)$ has the same law as $Z(t)$, and the second claim follows.
\end{proof}
\begin{rem} The above proof also shows that, if $n,l\in\Z_+^r$ and $l\nleq n$, then
$$\E  \left[ Y(t)^{-n}  Z(t)^{l}\right] = 0.$$
\end{rem}
\begin{cor}\label{pic}
$$a_r(n)=\lim_{t\to\infty}  \frac{1}{n!} \E Z(t)^{n}.$$
\end{cor}

Corollary~\ref{pic} yields a probabilistic interpretation of the (index zero) fundamental 
$SL(r+1,\R)$ Whittaker function, as follows.  
This complements a similar representation for class one Whittaker functions in 
terms of exponential functionals of Brownian motion given in \cite[Proposition 5.1]{boc}.
When $\alpha\equiv 0$, the fundamental $SL(r+1,\R)$ Whittaker function
$\phi_r(x)$ defined by \eqref{phir} may be written as $\phi_r(x)=\Phi_r(y)$ where
$y_i=e^{x_i-x_{i+1}}$, $i=1,\ldots,r$, and
$$\Phi_r(y) = \sum_{n\in\Z_+^r} a_r(n) y^n.$$
This series is well defined for any $y\in\C^r$.
\begin{cor}\label{ief2} For any $y\in\C^r$,
$$\Phi_r(y) = \lim_{t\to\infty} \E \exp\left(\sum_{k=1}^r y_k Z_k(t)\right) .$$
\end{cor}
Recalling Propositions \ref{a1-kmg} and~\ref{dd}, we also deduce the following.
\begin{cor} If $r=2$ and $T$ is an independent exponentially distributed random variable with parameter $\l^2$,
independent of $B$, then
$$\E Z_1(T)^n = \E Z_2(T)^n = \frac{1}{ (1+i\l)_n (1-i\l)_n} ,$$
$$\E \left[ Z_1(T)^n Z_2(T)^m\right]= {n+m\choose n} \E Z_1(T)^n \E Z_2(T)^m.$$
\end{cor}

We now consider the Markov chain on
$$E_r=\{n\in\Z_+^r:\ n_1\ge\cdots\ge n_r\ge 0\}$$
with generator
$$M^r = \sum_{i=1}^{r-1} n_i (n_i-n_{i+1}) D_{n_i} +n_r^2 D_{n_r} .$$
Recall that this describes the evolution of the first
row $n_i=\pi_{1,r-i+1}$, $i=1,\ldots,r$ of the Markov chain on $\Pi^r$ with generator
$G^r$ given by \eqref{gr} with $\alpha=0$.  Denote the transition kernel of
the Markov chain with generator $M^r$ by $q_t(n,m)$.
In the above notation, let us define
\be\label{ur}
U^r(t) = \int_{0<s_1<s_2<\cdots<s_r<t} Y_1(s_1)\ldots Y_r(s_r) ds_1\ldots ds_r.
\ee
For $p\in\Z_+$, let $p^r$ denote the element $(p,p,\ldots,p)\in E_r$.
\begin{prop}\label{ief3}
$$q_t(p^r,0) = \frac1{p!^r} \E U^r(t)^p.$$
\end{prop}
\begin{proof}
We first note the Markov duality relation $$M_n^r K(n,x) = \J_x^r K(n,x),$$ 
where, writing $x=a+ib$ with $a,b\in\R^{r+1}$,
$$\J^r = \frac12\sum_{j=1}^{r+1} \frac{\partial^2}{\partial a_j^2}
+\sum_{k=1}^r \Re y_k \frac{\partial}{\partial a_{k+1}} 
+\sum_{k=1}^r \Im y_k \frac{\partial}{\partial b_{k+1}} ,$$
$$K(n,x)= n!\, y^{-n}, \qquad y_k=e^{i(x_k-x_{k+1})},\quad k=1,\ldots, r .$$
The Markov process with generator $\J^r$, started at $x\in\C^{r+1}$, may be realised via the system of SDE's 
$dX_1=dB_1$ and, for $k=1,\ldots,r$, $$dX_{k+1}=dB_{k+1}+e^{i(X_k-X_{k+1})} dt.$$
By the above Markov duality relation, for example using \cite[Theorem 4.4.11]{ek}, 
if $N(t)$ is a Markov chain with generator $M^r$
started from $n\in E_r$, then
$$\E K(N(t),x) = \E K(n,X(t)).$$
The left hand side is given by
$$\E K(N(t),x) = \sum_{m\in E_r} m! y^{-m} q_t(n,m).$$
Let $n=p^r$. Then the right hand side is
$$\E K(n,X(t)) = p!^r \E e^{i p (X_{r+1}(t)-X_1(t))}.$$
To compute this, we first note that $X$ satisfies $e^{iX_1(t)}=e^{i(x_1+B_1(t))}$ and
$$e^{iX_{k+1}(t)}=e^{iB_{k+1}(t)}\left(e^{i x_{k+1}} + \int_0^t e^{i(X_{k}(s)-B_{k+1}(s))} ds\right),\qquad k=1,\ldots,r.$$
Iterating this identity, we obtain, for $k=1,\ldots,r$,
$$e^{iX_{k+1}(t)}=e^{iB_{k+1}(t)} \left(  \sum_{j=1}^{k} e^{i x_j} U^k_j(t) +e^{i x_{k+1}}  \right)$$
where
$$U^k_j(t) = \int_{0<s_j<\cdots<s_k<t} Y(s_j)\ldots Y(s_k) ds_j\ldots ds_k.$$
Note that $U^r(t)=U^r_1(t)$.
Thus,
$$\E e^{i p (X_{r+1}(t)-X_1(t))} = \E\left[ [Y_1(t)\ldots Y_r(t)]^{-p} \left( U^r(t) + \sum_{j=2}^r \frac1{y_1\ldots y_j} U^r_j(t) + \frac1{y_1\ldots y_r} \right)^p\right].$$
We can now apply bounded convergence and Cauchy's theorem to obtain
$$q_t(p^r,0)=p!^r \E [Y_1(t)\ldots Y_r(t)]^{-p} U^r(t)^p .$$
Since the law of Brownian motion is invariant under time reversal, this implies
$$q_t(p^r,0) = p!^r \E U^r(t)^p ,$$
as required.
\end{proof}

Combining this with Theorem~\ref{el-nu} yields the following.

\begin{cor}\label{ief4} Let $(N_1(t),\ldots,N_r(t)),\ t>0$ be a Markov chain with generator $L^r$ started from $+\infty$
and set $\zeta=\inf\{ t>0:\ N_1(t)=0\}$.  Then
$$\P(\zeta\le t) = \lim_{p\to\infty} \frac1{p!^r} \E U^r(t)^p.$$
\end{cor}

\begin{rem} The statements of Proposition \ref{ief1} and Corollaries~\ref{pic}
and~\ref{ief2} are easily extended to more general $\alpha\in\Z^r$.
For simplicity, suppose $\alpha\in\Z_+^r$.
Let $p_t(n,m)$ denote the transition kernel of the Markov chain on $\Z_+^{r}$
with generator $L^{r}$, as defined in Section~\ref{mrss} for $\alpha\in\Z_+^r$, and define
$$Y_k(t)=e^{i(B_k(t)-B_{k+1}(t))+\alpha_k t},\qquad Z_k(t)=\int_0^t Y_k(s) ds,\qquad k=1,\ldots, r.$$
Then, with $a_r(n)$ defined for $\alpha\in\Z_+^r$ as in Section 1.1, 
the {\em first} statement of Proposition~\ref{ief1} remains valid as stated.  
This follows from the generalised duality relation
$h^r_n F(n,x) = \I^r_x F(n,x)$, with $F(n,x)$ as before, $h^r$ defined by~\eqref{hr} and
$$\I^r=\frac12\sum_{j=1}^{r+1} \frac{\partial^2}{\partial x_j^2}-\sum_{j=1}^{r+1} i\nu_j \frac{\partial}{\partial x_j}+\sum_{k=1}^r e^{i(x_k-x_{k+1})},$$
where $\nu\in\R^{r+1}$ is such that $\sum_{i=1}^{r+1} \nu_i=0$ and $\alpha_i=\nu_i-\nu_{i+1}$, $1\le i\le r$.
For the second statement of Proposition~\ref{ief1}, let us denote $Z(t)=Z^{(\alpha)}(t)$.  Then, since
time reversal changes the sign of the drift $-i\nu$ and hence $\alpha$, the second claim becomes
$$\P_n(T_0\le t) = p_t(n,0) =  \frac{\E Z^{(-\alpha)}(t)^{n}}{a_r(n) n!}.$$
If $\alpha_i>0$ for each $i=1,\ldots,r$, then the random variable $Z^{(-\alpha)}(t)$ converges almost surely
as $t\to\infty$ to a limiting random variable $Z^{(-\alpha)}(\infty)$, and the statements of Corollaries~\ref{pic}
and~\ref{ief2} can be interpreted accordingly.  
For general $\alpha\in\Z^r$ the results also carry over but the precise statements need to
be modified as in Section~\ref{ia}.
\end{rem}

\begin{rem} More detailed results on imaginary exponential functionals of Brownian motion 
in the one dimensional case $r=1$ may be found in the paper~\cite{gdl}.
\end{rem}

\section{More general shapes}\label{rpp}

Let $\alpha_1,\alpha_2,\ldots$ be a sequence of non-negative integers and denote
$$\beta_{ij}=\alpha_i+\alpha_{i+1}+\cdots+\alpha_{i+j-1}.$$
Let $\Pi^{\l/\mu}$ denote the set of non-negative integer arrays $(\pi_{ij},\ (i,j)\in\l/\mu)$
which satisfy $\pi_{ij}\ge \pi_{i,j-1} \vee (\pi_{i-1,j}-\beta_{ij}),\quad (i,j)\in\l/\mu$, 
with the convention $\pi_{ij}=0$ for $(i,j)\notin\l/\mu$.  We will write $\Pi^\l=\Pi^{\l/\phi}$.
Note that if $\alpha \equiv 0$ then $\Pi^{\l/\mu}$ is the set of reverse plane
partitions of shape $\l/\mu$.

Recall that for a partition $\l$ we denote by $\l^\circ$ the 
set of $(i,j)\in\l$, such that $(i+1,j)\in\l$ and $(i,j+1)\in\l$.
Fix $\l$ and $\mu\subset \l^\circ$.
Denote by $\tilde\mu$ the extension of $\mu$ to include $(i,j)\in\l/\mu$
such that either $(i-1,j)\in\mu$ or $(i,j-1)\in\mu$.
For $\sigma\in\Pi^{\l/\mu}$, let $\Pi_\sigma^\l$ denote the set of $\pi\in\Pi^\l$ with $\pi |_{\l/\mu}=\sigma$.
For $\pi\in\Pi^\l$, let
$$W_{\l,\mu}(\pi)=\prod_{(i,j)\in\tilde\mu} {\pi_{ij}\choose \pi_{i,j-1}} {\pi_{ij}+\beta_{ij}\choose \pi_{i-1,j}}
 \prod_{(i,j)\in\mu} \frac{\pi_{ij}!}{(\pi_{ij}+\beta_{ij})!} .$$
For $\sigma\in\Pi^{\l/\mu}$, let $K^{\l,\mu}_\sigma$ denote the probability distribution on $\Pi_{\sigma}^\l$ 
defined by $K^{\l,\mu}_\sigma(\pi)=W_{\l,\mu}(\pi)/A_{\l,\mu}(\sigma)$, where
$$A_{\l,\mu}(\sigma)=\sum_{\pi\in\Pi_\sigma^\l} W_{\l,\mu}(\pi).$$
For $\sigma\in\Pi^{\l/\mu}$ and $(i,j)\in\l/\mu$, set
$$b_{ij}(\sigma)=(\sigma_{ij}-\sigma_{i,j-1})(\sigma_{ij}-\sigma_{i-1,j}+\beta_{ij}),$$
with the convention $\sigma_{ij}=0$ for $(i,j)\notin\l/\mu$.  
Let $C(\mu)$ denote the set of external corners of $\mu$, that is, the set of
$(i,j)\in\mu$ such that $(i,j+1)\notin\mu$ and $(i+1,j)\notin\mu$.  
Define
$$G^{\l,\mu} = \sum_{(i,j)\in\l/\mu} b_{ij}(\sigma) D_{\sigma_{ij}},\qquad G^\l=G^{\l,\phi}$$
$$H^{\l,\mu} = G^{\l,\mu} +V_{\l,\mu}(\sigma),\quad 
V_{\l,\mu}(\sigma)=\sum_{(i,j)\in C(\mu)}\sigma_{i+1,j}\sigma_{i,j+1} + \sum_{i=1}^{l(\mu)} \beta_{i+1,\mu_i} \sigma_{i,\mu_i+1}.$$

In proving the following theorem, we will also show that $H^{\l,\mu} A_{\l,\mu}=0$, 
so that the corresponding Doob transform 
$$L^{\l,\mu} = A_{\l,\mu}(\sigma)^{-1} \circ H^{\l,\mu} \circ A_{\l,\mu}(\sigma)$$
generates a Markov chain on $\Pi^{\l/\mu}$.

\begin{thm} \label{mf-rpp}
Let $\pi(t),\ t\ge 0$ be a Markov chain on $\Pi^\l$ with generator $G^\l$
and initial law $K^{\l,\mu}_\sigma$, for some $\mu\subset\l^\circ$ and $\sigma\in\Pi^{\l/\mu}$.
Then $\sigma(t)=\pi(t) |_{\l/\mu}$ is a Markov chain on $\Pi^{\l/\mu}$ with generator 
$L^{\l,\mu}$ and, for $t>0$, the conditional law of $\pi(t)$ given $\{\sigma(s),\ s\le t\}$
is $K^{\l,\mu}_{\sigma(t)}$.
\end{thm}
\begin{proof}
For notational convenience, let us fix $\l$ and $\mu\subset\l^\circ$ and write 
$W=W_{\l,\mu}$, $H=H^{\l,\mu}$ and $V=V_{\l,\mu}$.
For functions $f$ on $\Pi^\l$, define
$$(\Lambda f)(\sigma) = \sum_{\pi\in\Pi_\sigma^\l} W(\pi) f(\pi).$$
We will show that 
\be\label{gsi}
H \circ \Lambda = \Lambda \circ G^\l,
\ee
which implies both $H A_{\l,\mu}=0$ and the statement of the theorem.

For $\pi\in\Pi_\sigma^\l$ and $v\in\l/\mu$,
$$[L_{\sigma_v} W(\pi)] b_v(\sigma) = W(\pi) b_v(\pi) .$$
Thus,
\bea
[H (\Lambda f)](\sigma) &=&\sum_{\pi\in\rpp_\sigma(\l)} \left( [H W(\pi)] f(\pi)+
\sum_{v\in\l/\mu} [L_{\sigma_v} W(\pi)] b_v(\sigma) D_{\sigma_v} f (\pi) \right)\\
&=& \sum_{\pi\in\rpp_\sigma(\l)} \left( [H W(\pi)] f(\pi)+
\sum_{v\in\l/\mu} W(\pi) b_v(\pi) D_{\pi_v} f (\pi) \right) .
\eea

To complete the proof, we will show that, for $\pi\in\Pi_\sigma^\l$,
\be\label{hg}
H W(\pi) = {}^*\! G^{\mu} W(\pi),
\ee
where
$${}^*\! G^{\mu} = \sum_{u\in\mu} B_{\pi_u} \circ  b_u(\pi),$$
and $B_n$ denotes the forward difference operator $B_nf(n)=f(n+1)-f(n)$.

For $u=(i,j)\in\mu$, define
$$b_u'(\pi) = (\pi_{i,j+1}-\pi_{ij})(\pi_{i+1,j}-\pi_{ij}+\beta_{i+1,j}),$$
and note that
$$B_{\pi_u}(  b_u(\pi) W(\pi) ) = [b_u'(\pi)-b_u(\pi)] W(\pi)  .$$
On the other hand, for $\pi\in\Pi_\sigma^\l$ and $v\in\l/\mu$,
$$b_v(\sigma) D_{\sigma_v} W(\pi) = [b_v(\pi)-b_v(\sigma)] W(\pi).$$
Thus \eqref{hg} reduces to the identity, for $\pi\in\Pi_\sigma^\l$,
\be\label{ki} 
\sum_{v\in\l/\mu} [b_v(\pi)-b_v(\sigma)] + V(\sigma) = \sum_{u\in \mu} [b_u'(\pi)-b_u(\pi)] .
\ee
For $u=(i,j)$, let us write $u\to v$ is $v$ is either $(i+1,j)$ or $(i,j+1)$.
For $(i,j)\in\mu$,
\bea
b_{ij}'(\pi)-b_{ij}(\pi) &=&  \pi_{i+1,j}\pi_{i,j+1}-\pi_{ij}\pi_{i+1,j}-\pi_{ij}\pi_{i,j+1}\\
&& +\pi_{ij}\pi_{i-1,j}+\pi_{ij}\pi_{i,j-1}-\pi_{i-1,j}\pi_{i,j-1}\\
&& -\beta_{ij}(\pi_{ij}-\pi_{i,j-1}) +\beta_{i+1,j}(\pi_{i,j+1}-\pi_{ij}).
\eea
Summing over $(i,j)\in\mu$ gives
$$ \sum_{u\in \mu} [b_u'(\pi)-b_u(\pi)]  = \sum_{\substack{(i,j)\in\mu \\ (i+1,j+1)\notin \mu}} \pi_{i+1,j}\pi_{i,j+1}
- \sum_{\substack{u\in\mu, v\in\l/\mu \\ u\to v}} \pi_u\pi_v - T(\pi) + U(\pi),$$
where
$$T(\pi)=\sum_{i=1}^{l(\mu)}\beta_{i,\mu_i+1} \pi_{i,\mu_i} \qquad U(\pi) = \sum_{i=1}^{l(\mu)} \beta_{i+1,\mu_i} \pi_{i,\mu_i+1} .$$
On the other hand,
\begin{align*}
&\sum_{v\in\l/\mu} [b_v(\pi)-b_v(\sigma)] =
\sum_{\substack{(i,j)\in\l/\mu\\ (i-1,j-1)\in\mu \backslash C(\mu)}} \pi_{i-1,j}\pi_{i,j-1}
- \sum_{\substack{u\in\mu, v\in\l/\mu \\ u\to v}} \pi_u\pi_v - T(\pi) \\
&= \sum_{\substack{(i,j)\in\mu\\ (i+1,j+1)\notin \mu}} \pi_{i+1,j}\pi_{i,j+1} -  \sum_{(i,j)\in C(\mu)}\sigma_{i+1,j}\sigma_{i,j+1}
- \sum_{\substack{u\in\mu, v\in\l/\mu \\ u\to v}} \pi_u\pi_v - T(\pi) \\
&= \sum_{\substack{(i,j)\in\mu\\ (i+1,j+1)\notin \mu}} \pi_{i+1,j}\pi_{i,j+1} 
- \sum_{\substack{u\in\mu, v\in\l/\mu \\ u\to v}} \pi_u\pi_v - T(\pi) + U(\pi) -V(\sigma) ,
\end{align*}
as required.
\end{proof}

\begin{rem}\label{bat}
The main content of Theorem \ref{mf-rpp} is when $\tilde\mu=\l$.
In this case, when $\alpha\equiv 0$, the numbers $A_{\l,\mu}(\sigma)$ are, 
up to a trivial factor, the coefficients of the hypergeometric series 
of the partial flag manifold corresponding to $\mu$ as defined 
in~\cite[Definition 5.1.5, see also Theorem 5.1.6]{bcks}.
\end{rem}

\begin{rem}
Theorem \ref{mf-rpp} may be generalised to allow general integer-valued $\alpha_i$
by modifying the state spaces $\Pi^\l$ and $\Pi^{\l/\mu}$ as in the type $A_r$ case,
see \S\ref{ia}.  The basic symmetry observed there also extends naturally
to the general setting.
\end{rem}

\begin{prop}\label{ldp}
Let $\pi$ be distributed according to $K^{\l,\mu}_\sigma$, where $\sigma\in\Pi^{\l/\mu}$.
Suppose that in the limit as $N\to\infty$, $\sigma/N\to a$, where $a=(a_{ij})\in\left(\R_{>0}\right)^{\l/\mu}$.
Then, in the same limit, $\pi/N\to x^a$ in probability, where $x^a=(x^a_{ij})\in \left(\R_{>0}\right)^\l$ is the unique solution 
to the equations 
\be\label{cpF}
(x_{i+1,j}-x_{ij})(x_{i,j+1}-x_{ij})=(x_{ij}-x_{i-1,j})(x_{ij}-x_{i,j-1}),\qquad (i,j)\in \mu
\ee
satisfying $x_{ij}=a_{ij}$ for $(i,j)\in\l/\mu$, $0\le x_{ij}\le x_{i+1,j}\wedge x_{i,j+1}$ for $(i,j)\in\mu$,
and with the convention $x_{i,0}=x_{0,j}=0$.  
\end{prop}
\begin{proof}
The proof is similar to that of~\cite[Theorem 10.2]{rietsch},
see also~\cite[Lemma 5.4]{abo22}.
Let $X_a$ be the set of $x=(x_{ij})\in \R_+^\l$ satisfying $x_{ij}=a_{ij}$ for $(i,j)\in\l/\mu$
and $x_{ij}\le x_{i+1,j}\wedge x_{i,j+1}$ for $(i,j)\in\mu$.
By Stirling's formula, for $b\in\Z_+$,
$$\lim_{N\to\infty} \frac1{N} \log {yN+b \choose xN} = - y\; h\left(\frac{x}{y}\right),$$
uniformly on any compact set $0\le x\le y\le K$, where
$$h(p)=p\log p+(1-p)\log(1-p).$$
Thus,
\be\label{F-lim}
\lim_{\pi/N\to x} \frac1{N} \log W_{\l,\mu}(\pi) = -F(x),
\ee
uniformly for $x\in X_a$, where
$$F(x) = \sum_{(i,j)\in\mu} \left[ x_{i+1,j} h\left(\frac{x_{ij}}{x_{i+1,j}}\right) 
+ x_{i,j+1} h\left(\frac{x_{ij}}{x_{i,j+1}}\right) \right] .$$
We will show that $F$ is strictly convex on $ X_a^\circ$.
Let us write $u\to v$ if $u=(i,j)$ and $v$ is either $(i+1,j)$ or $(i,j+1)$.  
In this notation,
$$F(x) = \sum_{u\in\mu} \sum_{v\in\l : u\to v} x_v h\left(\frac{x_u}{x_v}\right) .$$
Let $X_a^\circ$ denote the set of $x=(x_{ij})\in \R_+^\l$ satisfying $x_{ij}=a_{ij}$ for $(i,j)\in\l/\mu$
and $0<x_{ij} < x_{i+1,j}\wedge x_{i,j+1}$ for $(i,j)\in\mu$.
Suppose that $x\in X_a^\circ$ so that $x_u>0$ for all $u\in\l$
and $x_v>x_u$ for all $u,v\in\l$ with $u\to v$.
For $u\in \mu$, we compute
$$\partial_{x_u} F(x) = \sum_{\substack{t\in\l\\ t\to u}} \log\left( 1-\frac{x_t}{x_u}\right)
- \sum_{\substack{v\in\l\\ u\to v}} \log \left( \frac{x_v}{x_u}-1\right)  ,$$
$$\partial_{x_u}^2 F(x) = \sum_{\substack{t\in\l\\ t\to u}} \frac{x_t}{x_u}\frac1{x_u-x_t} 
+ \sum_{\substack{v\in\l\\ u\to v}} \frac{x_v}{x_u}\frac1{x_v-x_u} .$$
For $u,v\in\mu$ with $u\to v$,
$$\partial_{x_u} \partial_{x_v} F(x) = - \frac1{x_v-x_u}.$$
If $u,v\in\mu$ and neither $u\to v$ nor $v\to u$, then $\partial_{x_u} \partial_{x_v} F(x) =0$.
The associated quadratic form is thus given by
$$\sum_{u,v\in\mu} \xi_u\xi_v \partial_{x_u} \partial_{x_v} F(x) =
\sum_{\substack{u,v\in\l \\ u\to v}} \frac1{x_v-x_u} \left( \frac{x_v}{x_u}\xi_u^2 + \frac{x_u}{x_v} \xi_v^2 - 2\xi_u\xi_v\right) ,$$
with the convention $\xi_u=0$ for $u\in\l/\mu$.  This is clearly non-negative
and, moreover, vanishes if, and only if,
$$\left( \frac{x_v}{x_u}\xi_u^2 + \frac{x_u}{x_v} \xi_v^2 - 2\xi_u\xi_v\right) = 
\left( \sqrt{\frac{x_v}{x_u}}\xi_u - \sqrt{\frac{x_u}{x_v}}\xi_v \right)^2 = 0,$$
or, equivalently, $\xi_u/x_u=\xi_v/x_v$, for all $u,v\in\l$ with $u\to v$.
Recalling that $\xi_v=0$ for $v\in\l/\mu$, this implies that
$\xi_u=0$ for all $u\in\mu$.  The Hessian is therefore positive
definite on $X_a^\circ$.

Now, $F$ is continuous on the compact set $X_a$ 
and therefore has at least one global minimiser $x^a\in X_a$.
If $x^a\in\partial X_a$ then, since $a_v>0$ for $v\in\l/\mu$,
there exists $u=(i,j)\in\mu$ such that either
$$x^a_{i-1,j}\vee x^a_{i,j+1} = x^a_{ij} < x^a_{i+1,j} \wedge x^a_{i,j+1}$$
or
$$x^a_{i-1,j}\vee x^a_{i,j+1} < x^a_{ij} = x^a_{i+1,j} \wedge x^a_{i,j+1},$$
with the convention $x^a_{ij}=0$ for $(i,j)\notin \N^2$.
In the first case, $\partial_{x_u} F(x^a)=-\infty$ and in the second case
$\partial_{x_u} F(x^a)=+\infty$, both contradicting the minimising
property of $x^a$.  Thus $x^a\in  X_a^\circ$ and,
by the strict convexity of $F$ on $ X_a^\circ$, it is the unique
minimiser and hence also the unique solution in $ X_a^\circ$ 
to the critical point equations \eqref{cpF}.  Finally, it is easy to
see that the critical point equations cannot hold at a point
on the boundary $\partial X_a$, since $a_v>0$ for $v\in\l/\mu$.

Since $F$ is continuous and the limit \eqref{F-lim} holds uniformly
for $x\in X_a$, it follows that the sequence $\pi/N$ satisfies a large deviation
principle in $X_a$ with rate function $I_a(x)=F(x)-F(x^a)$, hence
the statement of the proposition. 
\end{proof}

For $\rho,\pi\in\Pi^{\l}$, write $\rho\le\pi$ if $\rho_{ij}\le\pi_{ij}$, for all $(i,j)\in\l/\mu$.  Let $\Pi^{\l,*}$ 
denote the set of $\pi\in\Pi^{\l}$ whose entries may take the value $+\infty$
while still respecting the required inequalities $\pi_{ij}\ge \pi_{i,j-1}\vee (\pi_{i-1,j}-\beta_{ij})$.  
We will write $\pi\to+\infty$ (resp. $\pi=+\infty$) to mean $\pi_{ij}\to+\infty$ (resp. $\pi_{ij}=+\infty$)
for all $(i,j)\in\l/\mu$.

\begin{thm}\label{el} The Markov chain on $\Pi^\l$ with generator $G^\l$ has a unique
entrance law starting from $\pi=+\infty$.  Moreover, under this entrance law, 
for each $\mu\subset\l^\circ$, $\sigma(t)=\pi(t)|_{\l/\mu}$ is a Markov chain on $\Pi^{\l/\mu}$ with 
generator $L^{\l,\mu}$ and, for all $t>0$, the conditional law of $\pi(t)$, given $\{\sigma(s),\ s\le t\}$, 
is $K^{\l,\mu}_{\sigma(t)}$.
\end{thm}
\begin{proof}
The proof of the first claim is in two steps.  
First we note the following monotonicity property.
Given two different starting positions $\rho(0),\pi(0)\in\Pi^\l$ with
$\rho(0) \le\pi(0)$, it is clear from the definition of $G^\l$ that 
we may construct a coupling between two realisations $\rho(t)$ and 
$\pi(t)$ of the Markov chain with generator $G^\l$ and these starting positions,
such that, almost surely, $\rho(t)\le\pi(t)$ for all $t>0$.  Indeed, we simply allow 
the jumps at $(i,j)$ to occur independently unless $\rho_{ij}=\pi_{ij}$, in which case we
note that $b_{ij}(\rho)\ge b_{ij}(\pi)$ and couple the next jump so that
either $\rho_{ij}$ decreases by one or both $\rho_{ij}$ and $\pi_{ij}$
decrease by one, thus preserving the order $\rho\le\pi$.  It follows
that the law of the process $\pi(t),\ t\ge 0$ is stochastically increasing in 
the initial position $\pi(0)$.  
We can therefore let $\pi(0)\to+\infty$ in $\Pi^\l$ to obtain a unique (in law) limiting process 
$\pi(t),\ t\ge 0$ in $\Pi^{\l,*}$.
It only remains to show that $\pi(t)\in\Pi^\l$ for all $t>0$, almost surely.
This is straightforward, by induction on $\l$.

For the second claim, let $\tau,\sigma\in\Pi^{\l/\mu}$ with $\tau\le\sigma$.
By Theorem~\ref{mf-rpp}, we can construct random starting positions $\rho(0),\pi(0)$
taking values in $\Pi^\l$ such that: the distribution of $\pi(0)$ is $K^{\l,\mu}_\sigma$; 
the distribution of $\rho(0)$ is $K^{\l,\mu}_\tau$; 
$\rho(0)\le\pi(0)$, almost surely.  Indeed, we start with a realisation
$\pi(0)$ distributed according to $K^{\l,\mu}_\sigma$; then let this evolve according to 
$G^\l$ for a fixed time $s>0$ and condition on its restriction to $\l/\mu$
being $\tau$ at time $s$.  By Theorem~\ref{mf-rpp}, this construction
has the required properties.  The monotonicity property therefore
extends to such random initial conditions, as follows:
the law of the process $\pi(t),\ t\ge 0$, with initial distribution $K^{\l,\mu}_\sigma$, 
is stochastically increasing in $\sigma$.
By uniqueness of the entrance law, it remains to show that
if $\pi$ is distributed according to $K^{\l,\mu}_\sigma$, then $\pi\to+\infty$ in 
probability as $\sigma\to+\infty$.  Since $K^{\l,\mu}_\sigma$ is stochastically 
increasing in $\sigma$, it suffices to show this if, say, $N\to\infty$ and 
$\sigma_{ij}/N\to 1$ for all $(i,j)\in\l/\mu$.  This follows from Proposition~\ref{ldp}. 
\end{proof}

\section{Extensions to other root systems}\label{rs}
\subsection{Type $B_r$}

Define $H^{B_1}=n^2 D_n/2$ and, for $r\ge 2$,
$$H^{B_r} = \sum_{i=1}^{r-1} n_i^2 D_{n_i} + \frac12 n_r^2 D_{n_r} 
+\sum_{i=1}^{r-1} n_i n_{i+1}.$$
Denote by $B_r(n)$ the unique solution to $H^{B_r} B_r=0$ on $\Z_+^r$ with $B_r(0)=1$.
The numbers $B_r(n)$ are, up to a trivial factor, coefficients of a fundamental Whittaker
function associated with the group $SO_{2r+1}(\R)$.  A recursive (over $r$)
formula for these coefficients is given in~\cite{ishii}, and may be interpreted
as providing a formula for $B_r(n)$ as a sum over reverse plane partitions, as follows. 

Let $\delta_r'$ denote the shifted staircase shape
$$\delta_r' = \{(i,j):\ 1\le i\le j\le 2r-i\}.$$
For $\pi\in\rpp(\delta'_r)$, define
$$W_{B_r}(\pi)=\prod_{i=1}^{r-1} {\pi_{i+1,i+1} \choose \pi_{i,i} }
\prod_{(i,j)\in\delta'_r} {\pi_{i,j} \choose \pi_{i,j-1} } {\pi_{i,j} \choose \pi_{i-1,j} },$$
with the convention $\pi_{i,0}=\pi_{0,j}=0$.  
For $n\in\Z_+^r$, denote
by $\rpp_n(\delta'_r)$ the set of $\pi\in\rpp(\delta'_r)$ with $\pi_{i,2r-i}=n_i$,
$i=1,\ldots, r$.  
In this notation, Theorem 3.1 of \cite{ishii} yields the formula
\be\label{Brf1} B_r(n)=\sum_{\pi\in \rpp_n(\delta'_r)} W_{B_r}(\pi).\ee

Since $H^r B_r=0$ and $B_r>0$ on $\Z_+^r$, we may consider the Doob transform
$$L^{B_r}=B_r(n)^{-1} \circ H^{B_r} \circ B_r(n) = \sum_{i=1}^{r-1} \frac{B_r(n-e_i)}{B_r(n)} n_i^2 D_{n_i}+\frac12 \frac{B_r(n-e_r)}{B_r(n)} n_r^2 D_{n_r}.$$
For $\pi\in\rpp(\delta'_r)$, we define
$$b_{ij}(\pi)=
\begin{cases} (\pi_{ij}-\pi_{i-1,j})(\pi_{ij}-\pi_{i,j-1}) & i\ne j\\
 (\pi_{ii}-\pi_{i-1,i})(\pi_{ii}-\pi_{i-1,i-1})/2 & i=j\end{cases}$$
with the convention $\pi_{i,0}=\pi_{0,j}=0$,
and set
$$G^{B_r} = \sum_{(i,j)\in\delta_r'} b_{ij}(\pi) D_{\pi_{ij}}.$$
For $n\in\Z_+^r$, let $K^{B_r}_n$ be the probability distribution on 
$\rpp_n(\delta'_r)$ defined by $K^{B_r}_n(\pi)=W_{B_r}(\pi)/B_r(n)$.

\begin{thm}
Let $\pi(t),\ t\ge 0$ be a Markov chain on $\rpp(\delta_r')$ with generator $G^{B_r}$
and initial distribution $K^{B_r}_n$, for some $n\in\Z_+^r$.
Then $$N(t)=(\pi_{1,2r-1}(t),\ldots,\pi_{r,r}(t))$$ is a Markov chain on $\Z_+^r$ with generator 
$L^{B_r}$
and, for all $t>0$, the conditional law of $\pi(t)$ given $\{N(s),\ s\le t\}$
is $K^{B_r}_{N(t)}$.
\end{thm}
\begin{proof}
For $n\in\Z_+^r$ and functions $f$ on $\rpp(\delta'_r)$, define
$$(\Lambda_{B_r} f)(n) = \sum_{\pi\in \rpp_n(\delta'_r)} W_{B_r}(\pi) f(\pi).$$
We will show that $H^{B_r} \circ \Lambda_{B_r} = \Lambda_{B_r} \circ G^{B_r}$,
from which the statement of the theorem follows.
Note that this intertwining relation also implies $H^{B_r} B_r=0$, 
for $B_r$ defined by \eqref{Brf1}.  
The proof is similar to the proof of Theorem~\ref{mf-rpp}.

For $n\in\Z_+^r$, let $b_{i,2r-i}(n)=n_i^2$, $i=1,\ldots,r-1$ and $b_{rr}(n)=n_r^2/2$.  
Note that
$$H^{B_r} = \sum_{i=1}^r b_{i,2r-i}(n) D_{n_i} +\sum_{i=1}^{r-1} n_i n_{i+1}.$$
For $\pi\in\rpp_n(\delta_r')$ and $i=1,\ldots,r$, we have
$$[L_{n_i} W_{B_r}(\pi)] b_{i,2r-i}(n) = W_{B_r}(\pi) b_{i,2r-i}(\pi) .$$
Thus, as in the type $A$ case, it suffices to show that,
for $\pi\in\rpp_n(\delta_r')$,
\be\label{hgB}
H^{B_r} W_{B_r}(\pi) = \sum_{1\le i\le j <2r-i} B_{\pi_{ij}}\left(  b_{ij}(\pi) W_{B_r}(\pi) \right).
\ee
For $u=(i,j)$, with $1\le i\le j <2r-i$, define
$$b_u'(\pi) = \begin{cases} (\pi_{i+1,j}-\pi_{ij})(\pi_{i,j+1}-\pi_{ij}) & i\ne j\\
(\pi_{i+1,i}-\pi_{ii})(\pi_{i,i+1}-\pi_{ii})/2 & i=j \end{cases}$$
and note that
$$B_{\pi_u}(  b_u(\pi) W_{B_r}(\pi) ) = [b_u'(\pi)-b_u(\pi)] W_{B_r}(\pi)  .$$
On the other hand, for $\pi\in\rpp_n(\delta_r')$ and $i=1,\ldots,r$,
$$b_{i,2r-i}(n) D_{n_i} W_{B_r}(\pi) = [b_{i,2r-i}(\pi)-b_{i,2r-i}(n)] W_{B_r}(\pi).$$
Thus \eqref{hgB} reduces to the identity, for $\pi\in\rpp_n(\delta_r')$,
$$
\sum_{i=1}^r [b_{i,2r-i}(\pi)-b_{i,2r-i}(n)] + \sum_{i=1}^{r-1} n_i n_{i+1} = \sum_{1\le i\le j <2r-i} [b_{ij}'(\pi)-b_{ij}(\pi)] .
$$
This is readily verified, as in the type $A$ case.
\end{proof}

\begin{rem} The diagonal values $b_n=B_2(n,n)$ are the Ap\'ery numbers
$$b_n=\sum_k {n \choose k}^2 {n +k \choose k}$$
associated with $\zeta(2)$.  This sequence satisfies the recurrence
$$n^2 b_n = (11n^2-11n+3)b_{n-1}+(n-1)^2 b_{n-2},$$
with $b_0=1$ and $b_1=3$. We note that these Ap\'ery numbers
are also given, in the notation of the previous section, by
 the diagonal values $A_{(2,2,1),(1,1)}(n,n,n)$.
\end{rem}

\subsection{Type $BC_r$}

There is a more refined structure incorporating types $BC_r$ which
naturally interpolate, via intertwining relations, between the root systems of type $B_r$.
The intertwining relations we discuss here are analogous to those 
presented in~\cite{glo} in the context of class one Whittaker functions.

For example, if we let
$$H_k^{BC_1} = \frac12 k D_k + \frac12 k(k-1) D^{(2)}_k$$
where $D^{(2)}_k f(k)=f(k-2)-f(k)$, and 
$$Q^{B_1}_{BC_1}(n,k)=2^{-n} {n\choose k},$$
then one can easily check that
$$H^{B_1} \circ Q^{B_1}_{BC_1} = Q^{B_1}_{BC_1} \circ H^{BC_1}.$$
This intertwining relation yields the following.  Let $E=\{n\ge k\ge 0\}$.
\begin{prop}
Let $(X_t,Y_t)$ be a Markov chain on $E$ with generator
$$G=H_k^{BC_1} + \frac12 n(n-k) D_n.$$
Suppose $X_0=n$ and $Y_0 \sim \mbox{\rm Binomial}(n,1/2)$.
Then $X_t$ is a Markov chain on $\Z_+$ with generator
$H^{B_1}=n^2 D_n/2$ and, for all $t>0$, the conditional law of $Y_t$
given $X_s,\ 0\le s\le t$, is $\mbox{\rm Binomial}(X_t,1/2)$.
\end{prop}

For $r=2$, we define
$$H^{BC_2} = n^2 D_n + \frac12 m D_m + \frac12 m(m-1) D^{(2)}_m+nm,$$
and denote by $BC_2(n,m)$ the unique solution to $H^{BC_2} BC_2=0$ on $\Z_+^2$ with $BC_2(0,0)=1$.  
The numbers $BC_2(n,m)$ are positive integers and given by
$$BC_2(n,m)=\sum_{k} {n \choose k} {m \choose k} 2^k.$$
These are the Delannoy numbers.
They also satisfy $\tilde H^{BC_2} BC_2=0$, where
$$\tilde H^{BC_2} = \frac12 n^2 D_n + \frac12 m^2 D_m + \frac12 nm D_{n,m}+nm,$$
and $D_{n,m}f(n,m)=f(n-1,m-1)-f(n,m)$.
These claims follow from the intertwining relations
$$H^{BC_2}\circ Q = Q \circ H^{B_1},\quad \tilde H^{BC_2}\circ Q = Q \circ H^{B_1},\quad Q((n,m),k) = {n \choose k} {m \choose k} 2^k.$$
Denote the corresponding Doob transforms by
$$L^{BC_2} = BC_2(n,m)^{-1} \circ H^{BC_2} \circ BC_2(n,m),$$
$$\tilde L^{BC_2} = BC_2(n,m)^{-1} \circ \tilde H^{BC_2} \circ BC_2(n,m).$$
Let $P$ be the set of $(n,m,k)\in\Z_+^3$ satisfying $0\le k\le n\wedge m$ and
let $P_{n,m}$ be the set of $(a,b,c)\in P$ with $a=n$ and $b=m$.
Denote by $K^{BC_2}_{n,m}$ the probability distribution supported on $P_{n,m}$
and defined by $K^{BC_2}_{n,m}(n,m,k)=Q((n,m),k) /BC_2(n,m)$.  
Let
$$G=n(n-k)D_n+\frac12 (m-k) D_m +\frac12 (m-k)(m-k-1) D^{(2)}_m + \frac12 k^2 D_k,$$
$$\tilde G = \frac12 n(n-k)D_n + \frac12 m(m-k) D_m + (n-k)(m-k) D_{n,m} + \frac12 k^2 D_k.$$
The above intertwining relations yield the following.
\begin{prop} 
Suppose that $X=(X_1,X_2,X_3)$ is a Markov chain in $P$ with initial law $K^{BC_2}_{n,m}$
and generator $G$ (resp. $\tilde G$).
Then $(X_1,X_2)$ is a Markov chain with generator $L^{BC_2}$ (resp. $\tilde L^{BC_2}$)
and, for all $t>0$, in both cases, the conditional law of $X(t)$
given $(X_1(s),X_2(s)),\ 0\le s\le t$, is $K^{BC_2}_{X_1(t),X_2(t)}$.
\end{prop}

\subsection{Type $G_2$}

Let
$$H^{G_2}=n^2 D_n+3m^2 D_m+3nm$$
and denote by $G_2(n,m)$ the solution to $H^{G_2} G_2=0$ on $\Z_+^2$ with $G_2(0,0)=1$.  
The numbers $G_2(n,m)/(n!^2m!^2)$ are the series coefficients of the fundamental Whittaker
function, with index zero, associated with the group $G_2(\R)$~\cite{h,ishii}.

Let $\Pi$ denote the set of $(n,m,i,j,k,l)\in\Z_+^6$ satisfying 
$$k \le i\wedge j, \quad i\vee j \le l, \quad l\le n\wedge m,\quad i+j\le n.$$
For $(n,m,i,j,k,l)\in\Pi$, set
$$W(n,m,i,j,k,l) ={n \choose i,j} {n\choose l}{m\choose l}{l\choose i}{l\choose j}{i\choose k}{j\choose k}.$$
For functions $f$ on $\Pi$, define
$$(\Lambda f)(n,m) = \sum_{i,j,k,l} W(n,m,i,j,k,l) f(n,m,i,j,k,l).$$
Let
\begin{align*}
G=(n-l)(n-i-j) D_n + 3m&(m-l) D_m + 3(l-i)(l-j) D_l\\
&+ i(i-k) D_i + j(j-k) D_j +k^2 D_k.
\end{align*}
Then one can check that
\be\label{ir-G2e}
H^{G_2}\circ \Lambda =  \Lambda \circ G.
\ee
This immediately yields the binomial sum formula
$$G_2(n,m)=\sum_{i,j,k,l} W(n,m,i,j,k,l),$$
which may be simplified to obtain, for example,
\be\label{g2-bs}
G_2(n,m)=\sum_{i,j} {n \choose i} {n \choose j}  {m \choose i} {m \choose j}  {n+m-i-j \choose m} {i+j \choose j}.
\ee
One can check that this agrees with \cite[Theorem 5.1]{ishii}.  

Since $H^{G_2} G_2=0$ and $G_2>0$ on $\Z_+^2$, the corresponding Doob transform
$$L^{G_2} = G_2(n,m)^{-1} \circ H^{G_2} \circ G_2(n,m)$$
generates a Markov chain on $\Z_+^2$.
For $(n,m)\in\Z_+^2$, let $K_{n,m}$ denote the probability distribution on $\Pi$ 
which is supported on the set of $(p_1,p_2,\ldots,p_6)\in \Pi$ with $p_1=n$ and $p_2=m$,
and defined on this set by $$K_{n,m}(n,m,i,j,k,l)=W(n,m,i,j,k,l)/G_2(n,m).$$
The intertwining relation \eqref{ir-G2e} then yields the following.
\begin{prop}. Let $X=(X_1,X_2,\ldots,X_6)$ be a Markov chain on $\Pi$ with initial law $K_{n,m}$
and generator $G$.
Then $(X_1,X_2)$ is a Markov chain on $\Z_+^2$ with generator $L^{G_2}$ and
moreover, for all $t>0$, the conditional law of $X(t)$,
given $\{X_1(s),X_2(s), 0\le s\le t\}$, is $K^{G_2}_{X_1(t),X_2(t)}$.
\end{prop}

\begin{rem}
The coefficient $G_2(n,m)$ is the constant term of 
$P^nQ^m$, where
$$P= \frac{(1+x+y+xz)(xw+yz+yw)}{xyz} ,\qquad Q= \frac{(1+y+z+w)}{w}.$$
This is easily verified using \eqref{g2-bs}.  One can also check that the Newton polyhedron
of the Laurent polynomial $f=PQ$ is reflexive.  
Setting $x=z_1/z_0$, $y=z_2/z_0$, $z=z_3/z_0$ and $w=z_4/z_0$, 
the equation $f=\psi$ may be written as
$$(z_0^2+z_0 z_1+z_0 z_2+z_1 z_3)(z_1 z_4+z_2 z_3+z_2 z_4)(z_0+z_2+z_3+z_4)=\psi z_0z_1z_2z_3z_4,$$
and defines a family of quintic threefolds in $\P^4$.
We note that the constant term series coefficients of $f$, given by the diagonal values
$$G_2(n,n)=\sum_{i,j} {n \choose i}^2 {n \choose j}^2 {i+j \choose j} {2n-i-j \choose n},$$
agree with the holomorphic function coefficients associated with the Calabi-Yau 
equation listed as  \#212 in the database~\cite{aesz}.
\end{rem}

\end{document}